\documentclass[notitlepage,reqno,11pt]{amsart}
\usepackage{latexsym,amssymb, epsfig, amsmath,amsfonts, amsthm}
\usepackage{color}
\usepackage{multirow}
\usepackage{relsize}
\usepackage{microtype}
\usepackage{dsfont}
\usepackage{wasysym}
\usepackage{cases}
\usepackage{tikz}
\usepackage[colorlinks, allcolors=blue]{hyperref}

\usepackage[margin=1in]{geometry}

\numberwithin{equation}{section}
\newtheorem{theorem}{Theorem}[section]
\newtheorem{lemma}[theorem]{Lemma}
\newtheorem{assumption}{Assumption}[section]
\newtheorem{definition}[theorem]{Definition}
\newtheorem{coro}[theorem]{Corollary}
\newtheorem{prop}[theorem]{Proposition}
\newtheorem{remark}[theorem]{Remark}

\def\E{\mathbb{E}}

				\newcommand{\fF}{\overline{\mathfrak F}}
				\newcommand{\fS}{\overline{\mathfrak S} }
				
				\newcommand{\dD}{\mathbb{D}}
				\newcommand{\dE}{\mathbb{E}}
				
				\newcommand{\dP}{\nu}  
				\newcommand{\dR}{\mathbb{R}}

				\newcommand{\cB}{\mathcal{B}}
				\newcommand{\cC}{\mathcal{C}}

				\newcommand{\cG}{\mathcal{G}}
				\newcommand{\cH}{\mathcal{H}}

				\newcommand{\Cov}{\text{\rm Cov}}
				\newcommand{\wt}{\widetilde}

				\newcommand{\R}  {\mathbb{R}}

				\newcommand{\bC}{{\mathbf C}}

				\newcommand{\non}{\nonumber}

				\newcommand{\baa}{\begin{eqnarray*}}
					\newcommand{\eaa}{\end{eqnarray*}}

				\newcommand{\indic}[1]{\mathds{1}_{#1}}
				
				\newcommand{\SCA}[1]{{{\left<#1\right>}}} 
				\newcommand{\PAR}[1]{{{\left(#1\right)}}} 
				\newcommand{\SBRA}[1]{{{\left[#1\right]}}} 
				\renewcommand{\leq}{\leqslant}
				\renewcommand{\geq}{\geqslant}

				\providecommand{\keywords}[1]{\textbf{Keyswords: } #1}
				\newcommand{\rd}{\mathrm{d}}
				\newcommand{\hF}{\hat{\mathfrak {F}}}
				\newcommand{\hS}{\hat{\mathfrak{S}}}
				
\begin{document}

					\title[Functional Central Limit for Epidemic Model with Memory]{
					Functional Central Limit Theorem and SPDE for epidemic model with memory of the last infection and waning immunity}

					\author[Arsene Brice Zotsa Ngoufack]{Arsene Brice Zotsa Ngoufack}
					\thanks{ABZN acknowledges funding from the SCOR Foundation for Science.}
					\address{Departement de mathématiques, 
						Université du Québec à Montréal}
					\email{zotsa\_ngoufack.arsene\_brice.@uqam.ca}
					
%
					
					\date{\today}
					
					
					\begin{abstract}We study the fluctuations of a stochastic epidemic model with memory of previous infections, varying infectivity, and waning immunity, as introduced in \cite{guerin2025stochastic}. The dynamics of the epidemic model are described by a measure-valued process with respect to infection age and individual traits.
						The Functional Law of Large Numbers (FLLN) is formulated as an integral equation, which is solved by a deterministic measure. In this article, we establish the Functional Central Limit Theorem (FCLT), capturing the fluctuations of the stochastic model around its deterministic limit. The limit of the FCLT is given by a nonlinear stochastic integral equation which is solved by a random signed-measure.
						We further derive the weak solution in the form of a stochastic partial differential equation (SPDE) and propose an alternative representation of the FCLT, as fluctuations in the average total force of infection and average susceptibility.
					\end{abstract}
					\maketitle
					\section{Introduction}
					The recent pandemic has shown the importance to continuous to improve epidemic models. Usually scientist use compartmental model to capture the dynamic of epidemic. In particular, the classical \textbf{SIRS} epidemic model, where $\textbf{S}$ represents the compartment of susceptible individuals, $\textbf{I}$ denotes infected individuals and $\textbf{R}$ corresponds to the recovered individuals, assumes that individuals who leave compartment \textbf{R} become instantly fully susceptible again\cite{britton2019stochastic,ethier2009markov}. This means that once an individual has recovered, they remains fully immunized during some period, after which they becomes immediately susceptible. In this model, we also assume that, individuals move from class \textbf{S} to  \textbf{I} at a constant rate called infectivity rate. After some period, they move to \textbf{R} and after another period, they return to \textbf{S}. In the case of stochastic epidemic model, when the duration in each compartment follows an exponential distribution, the model is said to be Markovian \cite{britton2019stochastic}. Additionally, we assume that, the population is homogeneous, meaning that, it does not take into account inhomogeneities such as spatial type or social activity.  The study of this model in large population has been extensively conducted, demonstrating that the classical ordinary differential equation (ODE) epidemic model can be seen as the limit of a Markov process. This result is known as a Functional Law of Large numbers (FLLN)\cite{britton2019stochastic,ethier2009markov}.   
					
					In \cite{pang2022functional}, Pang and Pardoux established the Functional Law of Large Numbers (FLLN) for the SIRS epidemic model, considering an arbitrary distribution of the duration of sojourn in each compartment while keeping the infectivity rate constant, making their model non-Markovian. Later, in \cite{forien_epidemic_2021}, together with Forien, they extended this result to the SIR epidemic model with a varying infectivity rate. In their paper, they assumed that the infectivity of each individual is a random function of the elapsed time since infection.
					In \cite{pang2023functional}, they proposed a similar model, but with an infectivity rate that depends on the age of infection. To study the fluctuations between the stochastic model and the results obtained from the FLLN, they also proved the Functional Central Limit Theorems (FCLT) in \cite{pang2022-CLT-functional, pang2024spde}.
					A spatial version of the model can be found in \cite{kanga2024spatial}, while a non-homogeneous model is discussed in \cite{pang2025stochastic}. Additionally, in \cite{foutel-rodier_individual-based_2020}, the authors proposed a similar model based on a structured-age branching process, known as the Crump-Mode-Jagers process \cite{crump1968general}.
					
					Recently, in \cite{forien-Zotsa2022stochastic}, Forien, Pang, Pardoux, and Zotsa-Ngoufack established a Functional Law of Large Numbers (FLLN) for a homogeneous stochastic epidemic model with varying infectivity, which take account the gradual loss of immunity. The Central Limit Theorem (FCLT) was later proved separately in \cite{ngoufack2025functional} by Zotsa-Ngoufack.
					In addition to the assumption that the infectivity of each individual is a random function, in their model they also assume that, the susceptibility of each individual is a random function of the elapsed time since infection. In their framework, at each new infection, a new pair of càdlàg random functions, independent of the previous ones, is drawn to define the infectivity and susceptibility of the individual. This implies that they did not consider the memory on the previous infection.
					Their model generalizes the deterministic model of Kermack-McKendrick \cite{KMK,kermack_contributions_1932,kermack_contributions_1933} and demonstrates that the Kermack-McKendrick model can be viewed as the limit of a stochastic model. The proof of their FCLT was particularly challenging due to the interactions and memory effects present in the system\cite{ngoufack2025functional}. In \cite{foutel2025optimal}, the authors extended the model from \cite{forien-Zotsa2022stochastic} by incorporating vaccination policies.
					The model in \cite{forien-Zotsa2022stochastic} was non-Markovian and posed significant analytical difficulties. 
					
					In \cite{guerin2025stochastic}, Guérin and Zotsa-Ngoufack extended the model in \cite{forien-Zotsa2022stochastic}. They introduced an age structure and incorporated memory from previous infections into the stochastic model. In their formulation, the infectivity and susceptibility functions are parametric, where a parameter referred to as the trait describes the characteristics of the individual. Given a specific trait, the infectivity and susceptibility functions become deterministic for that individual.
					
					By incorporating an age structure into the epidemic model, the system becomes Markovian and more analytically tractable, benefiting from the existing literature on measure-valued processes. Specifically, for the FLLN, the authors established convergence results for the empirical measure of infection age and individual traits. The age and trait process are described by a piecewise deterministic Markov process (PDMP). To incorporate memory from previous infections, the authors used a transition kernel defined over the space of individual traits.
					
					In this article, we establish a FCLT for the model introduced in \cite{guerin2025stochastic}. In contrast to \cite{ngoufack2025functional} we use Sobolev spaces to tackle the tightness of the fluctuation of the empirical measure around its deterministic limit. We also see that adding the structure in the model, it becomes more easier to analyze the interaction between individuals compared to the approach in \cite{ngoufack2025functional}. However we still introduce a process which counts the number of times of (re-)infection of each individual to tackle the interactions. 
					Unlike in \cite[Assumption~$2.4-2.6$]{ngoufack2025functional}, we do not impose additional assumption on the infectivity and susceptibility functions. Instead, we assume that, the initial age of infection has a bounded moment (Assumption~\ref{VVM-CLT-ass-a_0}). 
					For tightness we apply Aldou's criterion for Hilbert space (see Definition~\ref{Def-tight}). We then deduce the proof of the FCLT by establishing that the limit of any subsequence do not depend of the subsequence and we conclude by continuous mapping theorem.  The limit of this FCLT is given by a  signed-measured which is a solution of a stochastic integral equation (Theorem~\ref{CLT-formulation-VVM}).
					If we assume that the signed-measure has a density, we derive from the integral equation that the weak solution is given by a SPDE with Gaussian noise (Proposition~\ref{SPDE-sec-prop}). Additionally, we obtain an alternative expression of this solution given by a system of a stochastic Volterra equation driven by a two dimensional Gaussian process with a well known covariance function (Proposition~\ref{eq-equiv-spde}) and we deduce that, this result is similar to the result obtained in \cite{ngoufack2025functional} when we remove the memory (taking transition kernel equal to one).  
					
					The proof of this FCLT closely follows the approach used in \cite{chevallier_fluctuations_2017,tran2006modeles}. However, we obtain a more general regularity result in Sobolev space compared to \cite{tran2006modeles}. In comparaison to \cite{chevallier_fluctuations_2017}, we obtain the same regularity in Sobolev spaces without assuming bounded initial ages. In fact, in \cite{chevallier_fluctuations_2017} the author assumes that the initial age is bounded, an assumption that is not realistic for epidemic model, as it would prevent to derive a classical epidemic model based on ODEs. In contrast to \cite{chevallier_fluctuations_2017} we use a conditional moment inequality to handle cases where the initial age is not bounded (Proposition~\ref{VVM-CLT-prop-inq}). Furthermore, the model in \cite{chevallier_fluctuations_2017} is not an epidemic model; it incorporates only age-structure without traits which are essential in our framework. 
					The model introduced in \cite{tran2006modeles} differs from the one considered in this article, as it describes a reproduction model where the birth rate is independent of the state of other individuals and the birth process is independent of the death process. On the other hand, we do not assume that the equivalent to the birth rate is differentiable(see \cite{ngoufack2024stochastic} for more details).

					\subsection*{Organization of the chapter}
					The rest of the paper is organized as follows. In Section~\ref{sec-model-v}, we describe the model in Section~\ref{sec-desc-model} and we recall the FLLN from \cite{guerin2025stochastic} in Section~\ref{sec-FLLN-v}. Next, in Section~\ref{Main-result} we present the main results of this article. The proof of our main result is presented in Section~\ref{CLT-proof}. More precisely, in Section~\ref{tight-secsub} we establish tightness, next in Section~\ref{sec-ch-limits} we characterize the limit of subsequence and in Section~\ref{CLT-Main-r} we establish our main result. 
					\paragraph{\bf Notations} We denote by $\dD(\dR_+,E)$ the Skorohod space of càdlàg functions with values in a space $E$.
					For a measured space $\PAR{E,\cG,\mu}$,  $L^1(\mu)$ is the set of integrable functions with respect to the measure $\mu$, and more generally $L^p(\mu)$ with $p\in[1,\infty]$ is the Lebesgue space with respect to the measure $\mu$. For any measurable function $f$, non-negative or in $L^1(\mu)$, we denote $\SCA{\mu,f}=\int f\rd \mu$. For a non-negative or integrable with respect to $\nu$ function $f$ defined on $\PAR{\dR_+\times\Theta,\cB(\dR_+)\otimes\cH}$, we define $\dE_\nu\SBRA{f(a)}=\int_\Theta f(a,\theta)\nu(\rd\theta)$ for $a\in\dR_+$.
					
					\section{Model and Results}\label{sec-model-v}
					\subsection{Model description}\label{sec-desc-model}
					We introduce a probability space $\PAR{\Theta,\cH,\nu}$, with $\Theta\subset \dR^d$ and $d\geq 1$. Let $(\lambda(\cdot,\theta), \gamma(\cdot,\theta))_{\theta\in\Theta}$ a family of deterministic non-negative functions defined on $\mathbb{R}^+$, where $\lambda(a,\theta)$ and $\gamma(a,\theta)$ respectively model the infectivity and the susceptibility at age $a$ of an individual with parameter $\theta$. 
					
					As in \cite{guerin2025stochastic} we assume that $(\lambda,\gamma)$ satisfy the following assumption. 
					\begin{assumption}\label{V-As-1} We assume that:
						\begin{enumerate}
							\item  There exists $\lambda_*>0$ such that for any $(a,\theta)\in\R_+\times\Theta$, $0\leq\lambda(a,\theta)\leq\lambda_*.$ 
							\item for any $(a,\theta)\in\R_+\times\Theta$, $0\leq\gamma(a,\theta)\leq1$.
						\end{enumerate}
					\end{assumption}

					We consider a population of fixed size $N$. For $k\in\{1,\ldots, N\}$, we denote by $(a_k^N(t))_{t\geq 0}$ the age process and by $\theta_k^N(t)$ the trait of the $k$-th individual at time $t$. We assume that $(a_0^k,\theta_0^k)_{1\leq k\leq N}$ are i.i.d random variables with distribution $\mu_0$ on $\dR^+\times\Theta$ modeling the initial age and parameter of each individual.

					Each time an individual is infected, its age jumps to $0$ and a new parameter is randomly chosen. The ages and parameters of the other individuals are not impacted. Between two infections, the ages of all the individuals in the population increase linearly and their parameters remain constant. 
					
					We introduce a family of independent Poisson random measures $\PAR{Q_k}_{k\geq1}$  on $\mathbb{R}^+\times\Theta\times\mathbb{R}^+$ with intensity $\rd z\dP(\rd \theta)\rd t$. 
					
					As in \cite{guerin2025stochastic}, we also introduce  a memory kernel $K:\Theta\times\Theta\to\R_+,$ satisfying the following.
					\begin{assumption}\label{V-As-2} $K$ is a nonnegative measurable function on $\Theta\times\Theta$ such that for any $\theta\in\Theta\subset\mathbb{R}^d,$
						\begin{equation}
							\int_\Theta K(\theta,\widetilde\theta)\dP(\rd \widetilde\theta)=1.
						\end{equation}
					\end{assumption}
					
					Then the family $(a^N_k,\theta^N_k)_{1\leq k\leq N}$ can be described as the solution of the following system of stochastic differential equations,
					\begin{align}\label{eq:SDE}
						\begin{cases}
							a^N_k(t)&=\displaystyle{a_0^k+t-\int_{0}^{t}\int_{\Theta}\int_{0}^{\infty}a^N_k(s^-)\mathds{1}_{\fF^N(s^-)\gamma^N_k(s^-)K(\theta^N_k(s^-),\widetilde\theta)\geq z}Q_k\PAR{\rd z,\rd\widetilde\theta,\rd s}}\\[0.3cm]
							\theta^N_k(t)&=\displaystyle{\theta_0^k+\int_{0}^{t}\int_{\Theta}\int_{0}^{\infty}\left(\widetilde\theta-\theta^N_k(s^-)\right)\mathds{1}_{\fF^N(s^-)\gamma^N_k(s^-)K(\theta^N_k(s^-),\widetilde\theta)\geq z}Q_k\PAR{\rd z,\rd\widetilde\theta,\rd s}}\\[0.3cm]
							\gamma^N_k(t)&=\gamma(a^N_k(t),\theta^N_k(t)),
						\end{cases}
					\end{align}
					where the force of infection in the population is given by
					\begin{equation}
						\fF^N(t)=\frac{1}{N}\sum_{k=1}^{N}\lambda\PAR{a^N_k(t),\theta^N_k(t)}.
					\end{equation}
We introduce the empirical measure $\mu ^N_t$ of ages and traits at time $t\geq 0$, defined by
\begin{equation}\label{H-mu}
	\mu ^N_t=\frac{1}{N}\sum_{k=1}^{N}\delta_{(a^N_k(t),\theta^N_k(t))}.
\end{equation}
We easily note that
\begin{equation*}
	\fF^N(t)=\langle\mu ^N_t,\lambda\rangle.
\end{equation*}
We observe that individuals are in interaction through the force of infection of the disease in the population. 
Indeed, the individual $k$ gets (re)-infected at time $t$ at rate $\fF^N(t)\gamma^N_k(t)$ and if it occurs, his age jumps to $0$ and he is assigned a new parameter according to the distribution $K( \theta^N_k(t^-),\widetilde\theta)\nu(\rd \widetilde\theta)$.
The family $(a_k,\theta_k)_{1\leq k\leq N}$ is a system of interacting piecewise deterministic Markov processes on the Skorohod space $\dD(\dR_+, \dR_+\times \Theta)$.

Note that we construct $(a_k^N,\theta_k^N)_{1\leq k\leq N}$ by induction on the jumps times. Assumption \ref{V-As-1}  implies that the rate of occurrence of new infection $\fF^N(t)\gamma_k^N(t)$ is bounded almost surely
by the constant $\lambda_\ast$. Consequently the jump times do not accumulate, and the above induction defines $(a_k^N(t),\theta_k^N(t)$ for all $t\geq0$.

\subsection{Some known results}\label{sec-FLLN-v}
We recall the following Functional Large Law of Numbers (FLLN) result from \cite{guerin2025stochastic}.
\begin{theorem}\label{VVM-th}
Under Assumption~\ref{V-As-1} and Assumption~\ref{V-As-2},  as $N\to\infty,\,\mu ^N$ converges to $\mu$ in distribution to a measure $\mu \in\dD(\R_+;\mathcal{P}(\R_+\times\Theta)),$ which is the unique solution of
\begin{equation}\label{VVM-eq1}
\langle\mu _t,f_t\rangle=\langle\mu _0,f_0\rangle + \int_0^t \langle\mu _s,\partial_af_s+\partial_sf_s\rangle  \rd s+\int_{0}^{t}\langle\mu _{s},\lambda\rangle\langle\mu _{s},Rf_{s}\rangle\rd s,
	\end{equation}
for any bounded function $f$ on $\R_+\times\R_+\times\Theta$, and of class $\cC^1$ with respect to the first two variables, where the operator $R$ is given by, 
\begin{equation}\label{eq:Rjump-term}
Rf(a,\theta)=\int_{\Theta}\left(f(0,\widetilde\theta)-f(a,\theta)\right)\gamma( a,\theta)K(\theta,\widetilde\theta)\dP(\rd\widetilde\theta).
\end{equation}
\end{theorem}
As noted in \cite{guerin2025stochastic}, when $\mu_0$ has a density $u_0$ with respect to measure $\rd a\dP(\rd\theta)$, the  weak solution $\mu_t$ of  \eqref{VVM-eq1} also admits a density $u_t.$ Moreover, $(u_t)_{t\geq 0}$ satisfies the following partial differential equation (PDE): for any $(a,\theta)\in\R_+\times\Theta,$
\begin{numcases}{}
\partial_t u_t(a,\theta)+\partial_a u_t(a,\theta)=-\fF(t)\gamma(a,\theta)u_t(a,\theta)\nonumber\\[0.3cm]
u(t,0,\theta)=\displaystyle{\fF(t)\int_{\R_+\times\Theta}\gamma(a,\widetilde\theta)K(\widetilde\theta,\theta)u_t(a,\widetilde\theta)\rd a\dP(\rd\widetilde{\theta})}\label{VVL-eq-1}\\[0.3cm]
u(0,a,\theta)=u_0(a,\theta)\nonumber\\[0.3cm]
\fF(t)=\displaystyle{\int_{\R_+\times\Theta}\lambda(a,\theta)u_t( a,\theta)\rd a\dP(\rd\theta).}\label{eq:def-F}
\end{numcases} 
We introduce the deterministic function 
\begin{align}
\fS (t,\theta)&:=\int_{\R_+\times\Theta}\gamma(a,\widetilde\theta)K(\widetilde\theta,\theta)u_t(a,\widetilde\theta)\rd a\dP(\rd\widetilde{\theta})\label{eq:def-S}.
\end{align}
Moreover, as mentioned again in \cite{guerin2025stochastic}, when $K(\theta,\wt\theta)=K(\wt\theta),$ Theorem~\ref{VVM-th} is identical to the FLLN established in \cite{forien-Zotsa2022stochastic} when $\mu_0$ has a density.
				
Finally from \cite{guerin2025stochastic} we have the following Remarks.
\begin{remark}
	For each $t\geq0, \fF(t)\leq \lambda_*$ and $\int_{\Theta}\fS(t,\theta)\dP(\rd\theta)\leq 1$.
\end{remark}

\begin{remark}\label{Rq-a-theta}
The solution $\mu$ of Equation~\eqref{VVM-eq1} is the distribution of  a pair of random processes $(a(t),\theta(t))_{t\geq 0}$ solution of the following stochastic differential system
\begin{align*}
\begin{cases}
\displaystyle{a(t)=a_0+t-\int_{0}^{t}\int_{\Theta}\int_{0}^{\infty}a(s^-)\mathds{1}_{\fF(s^-)\gamma(s^-)K(\theta(s^-),\widetilde\theta)\geq z}Q(\rd z,\rd \widetilde\theta,\rd s)}\\[0.3cm]
\displaystyle{\theta(t)=\theta_0+\int_{0}^{t}\int_{\Theta}\int_{0}^{\infty}\left(\widetilde\theta-\theta(s^-)\right)\mathds{1}_{\fF(s^-)\gamma(s^-)K(\theta(s^-),\widetilde\theta)\geq z}Q(\rd z,\rd\widetilde\theta,\rd s)}\\[0.3cm]
\gamma(t)=\gamma(a(t),\theta(t)),
\end{cases}
\end{align*}
where $(a_0,\theta_0)$ is a random variable with distribution $\mu_0$, $\fF$ is defined by \eqref{eq:def-F}, and $Q$ is a Poisson measure  on $\mathbb{R}_+\times\Theta\times\mathbb{R}_+$ with intensity $\rd z\dP(\rd \widetilde\theta)\rd t$. They also note that $\fF(t)=\E\left[\lambda(a(t),\theta(t))\right]$.
\end{remark}
	
Note that Theorem~\ref{VVM-th} means that in large population, the dynamic of the epidemic becomes deterministic.  		

\subsection{Main Results}\label{Main-result}	
The purpose of this section is to establish a Functional Central Limit Theorem (FCLT) for the fluctuation of the stochastic sequence around its deterministic limit. More precisely, we define the following fluctuation process:
\[\hat{\mu}^N=\sqrt{N}\left(\mu^N-\mu\right),\]
and we want to find the limiting law of $\hat{\mu}^N$.	
\begin{remark}
	 By classical central limit theorem's it follows that  $\hat{\mu}_0^N$ converges to $\hat{\mu}_0$ as $N\to\infty$ to a centered Gaussian variable such that for any regular function $\varphi$ and $\psi,$
	 \[\E \left[\langle\hat{\mu}_0,\varphi\rangle\langle\hat{\mu}_0,\psi\rangle\right]=Cov(\varphi(a_0,\theta_0),\psi(a_0,\theta_0)).\]
\end{remark}
We start to recall the following results on weighted Sobolev spaces. We use this for tightness.
\subsubsection{Weighted Sobolev spaces}We assume that $\nu$ is absolutely continuous measure with respect to the Lebesgue measure.
We introduce the following Sobolev spaces. 
For any $j\in\mathbb N$ and $\alpha\in\R_+,$ let us consider the space of all real function $g$ defined on $\R_+\times\Theta$ with partial derivative up to order $j$ such that,
\[\|g\|_{j,\alpha}=\left(\sum_{0\leq|\beta|\leq j}\int_{\R_+\times\Theta}\frac{|D^{\beta}g(a,\theta)|^2}{(1+a^{2\alpha})}\rd a\dP(\rd\theta)\right)^{\frac{1}{2}}<\infty,\]
where $|\beta|=\beta_0+\beta_1+\cdots+\beta_d$ with $\beta=(\beta_0,\beta_1,\cdots,\beta_d)$ and $D^{\beta}g(a,\theta)=\partial^{|\beta|} g/\partial a^{\beta_0}\partial\theta_1^{\beta_1}\cdots\partial\theta_d^{\beta_d}$. Let $\mathcal{W}^{j,\alpha}_0$ be the closure of the set of functions of class $\mathcal{C}^\infty$ with compact support on $\R_+\times\Theta$ for the  norm $\|\cdot\|_{j,\alpha}$. $\mathcal{W}^{j,\alpha}_0$ is a Hilbert spaces with norm $\|\cdot\|_{j,\alpha}$. We denote by $\left(\mathcal{W}^{-j,\alpha}_0,\|\cdot\|_{-j,\alpha}\right)$ its dual space.
				
Note that 
				\begin{equation*}\label{CLT-VVM-injection}
					\begin{aligned}
						&\text{\rm if }k'\geq k, \text{\rm then } \|\cdot\|_{k,\alpha}\leq \|\cdot\|_{k',\alpha} \text{\rm{ and }} \|\cdot\|_{-k',\alpha}\leq \|\cdot\|_{-k,\alpha}\\ 
						&\text{\rm if }\alpha'\geq \alpha, \text{\rm then } \mathcal{W}_0^{k,\alpha}\hookrightarrow \mathcal{W}_0^{k,\alpha'} \text{\rm and }\mathcal{W}_0^{-k,\alpha'}\hookrightarrow \mathcal{W}_0^{-k,\alpha}.
					\end{aligned}
				\end{equation*}
				Let $\mathcal{C}^{j,\alpha}$ be the space of function $g$ with continuous partial derivatives up to order $j$ and such that for all $\beta\in\mathbb{N}^{d+1},\,|\beta|\leq j,$ 
				\[\lim_{a\to\infty}\frac{|D^\beta g(a,\theta)|}{1+a^{\alpha}}=0.\]
				This space is normed with 
				\[\|g\|_{\mathcal{C}^{j,\alpha}}=\sum_{0\leq|\beta|\leq j}\sup_{(a,\theta)\in\R_+\times\Theta}\frac{|D^\beta g(a,\theta)|}{1+a^{\alpha}}.\]
				We denote by $\left(\mathcal{C}^{-j,\alpha},\|\cdot\|_{-j,\alpha}\right)$ its dual space. 
				
				We recall that $\mathcal{C}^{k}_b$ is the space of bounded function of class $\mathcal{C}^{k}$ with bounded derivatives of every order less than k. Note that $\mathcal{C}^{k}_b=\mathcal{C}^{k,0}$ as normed space and we denote by $\mathcal{C}^{-k,}_b$ its dual space. For $\alpha>\frac{1}{2}$, there exists $C>0$ such that $\|\cdot\|_{k,\alpha}\leq C\|\cdot\|_{\mathcal{C}^{k}_b}$.
				
				We recall the following Sobolev embedding (\cite[Proposition~$4.1.2$]{tran2006modeles} or \cite[Section~$2.1$]{fernandez1997hilbertian}):
				
				We set 
				\[m_d:=\left[\frac{d+1}{2}\right]+1,\]
				where $[\cdot]$ denoted the integer part. 
				\begin{itemize}
					\item $\mathcal{W}_0^{m+k,\alpha}\hookrightarrow \mathcal{C}^{k,\alpha}$ for $m\geq m_d,\,k\geq 0$ and $\alpha\in\R_+.$ i.e there exists $C>0$ such that 
					\begin{equation}\|\cdot\|_{\mathcal{C}^{k,\alpha}}\leq C \|\cdot\|_{m+k,\alpha}.\label{CLT-emb-1}\end{equation}
					\item $\mathcal{W}_0^{m+k,\alpha}\hookrightarrow_{H.S} \mathcal{W}_0^{k,\alpha+\beta}$ for $m\geq m_d,\,k\geq 0, \,\alpha\in\R_+$ and $\beta\geq m_d$ where $H.S$ means that the embedding is of Hilbert Schmidt type. Thus there exists $C>0$ such that 
					\begin{equation}\|\cdot\|_{k,\alpha+\beta}\leq C \|\cdot\|_{m+k,\alpha}.\label{CLT-emb-2}\end{equation}
				\end{itemize}
				Consequently for $m\geq m_d,\,k\geq 0,\,\beta\geq m_d,$ and $\alpha\in\R_+$ the following embedding hold:
				\begin{equation}\label{CLT-VVM-injection-1}
					\mathcal{W}_0^{-k,\alpha+\beta}\hookrightarrow \mathcal{W}_0^{-(m+k),\alpha}.
				\end{equation}
				We also have for $\alpha>\frac{1}{2},$
				\begin{equation}\label{Injection-w}
					\mathcal{W}_0^{-k,\alpha}\hookrightarrow \mathcal{C}_b^{-k}
				\end{equation}
				We refer to \cite{adams2003sobolev} for more results on Sobolev spaces. 
				\subsubsection{Main results}
				Let $\alpha>\frac{1}{2}$. We introduce the following definition.
				\begin{definition}\label{CLTM-def1}
					We set $W$ a continuous centered Gaussian process with values in $\mathcal{W}^{m_d+1,\alpha}$  with a covariance function given for all $\varphi,\,\psi\in \mathcal{W}^{m_d+1,\alpha}_0,$ by for $t,\,t'\in\R_+$,
					\begin{equation}\label{CLT-cov-w}
						\E\left[W_t(\varphi_t)W_{t'}(\psi_t)\right]=\int_0^{t\wedge t'}\langle \mu_s,\lambda\rangle \langle \mu_s,\wt R(\varphi_s,\psi_s)\rangle ds
					\end{equation}
					where 
					\[\wt R(\varphi_s,\psi_s)(a,\theta)=\int_{\Theta}\left(\varphi_s(0,\wt\theta)-\phi_s(a,\theta)\right)\left(\psi_s(0,\wt\theta)-\psi_s(a,\theta)\right)\gamma(a,\theta)K(\theta,\wt\theta)\nu(d\wt\theta)\]
				\end{definition}
				 We make the following Assumption. 
				\begin{assumption}\label{VVM-CLT-ass-a_0}
					\begin{enumerate}
						\item $\E\left[a_0^{2\alpha}\right]<\infty.$
						\item $\int_{\Theta}\sup_{\theta\in\Theta}K(\theta,\widetilde{\theta})\nu(d\widetilde{\theta})<\infty.$
					\end{enumerate}
				\end{assumption}
				Note that, the exponential distribution satisfies Assumption~\ref{VVM-CLT-ass-a_0}-$(1)$.
				
				The following theorem is our main result.
				\begin{theorem}\label{CLT-formulation-VVM} Under Assumption~\ref{V-As-1}, Assumption~\ref{V-As-2} and Assumption~\ref{VVM-CLT-ass-a_0},
					as $N\to\infty,\,\hat{\mu}^N$ converges to $\hat{\mu}$ in distribution in $D(\R_+,\mathcal{W}^{-(m_d+1),\alpha}_0),$ for any $\alpha>\frac{1}{2}$, where $\hat{\mu}$ is a continuous weak solution of 
					\begin{equation}\label{CLT-exp-VVM}
						\langle \hat{\mu}_t,\varphi_t\rangle=\langle \hat{\mu}_0,\varphi_0\rangle+\int_{0}^{t}\langle \hat{\mu}_s,\partial_s\varphi_s+\partial_a\varphi_s\rangle ds+\int_{0}^{t}\left(\langle \hat{\mu}_s,\lambda\rangle\langle \mu_s,R\varphi_s\rangle+\langle \mu_s,\lambda\rangle \langle \hat{\mu}_s,R\varphi_s\rangle\right)ds + W_t(\varphi_t),
					\end{equation}
					for any bounded measurable function $\varphi$ on $\R_+\times\R_+\times\Theta$ of class $\cC^1$ with respect to the first two variables such that $\varphi_t\in \mathcal{W}^{m_d+2,\alpha}_0$ and where $W$ is a Gaussian process with covariance function given by Definition~\ref{CLTM-def1}
				\end{theorem}
				We refer to  Section~\ref{CLT-proof} for the proof.
				
				\subsubsection{SPDE}
					We assume that $\hat\mu_0(da,d\theta)=\hat u_0(a,\theta)da\nu(d\theta)$. We define the process $(\hat{u}_t)_{t\geq0}$ belongs to $\mathcal{W}_0^{-(m_d+1),\alpha},$ as the solution of the following SPDE: for any $(a,\theta)\in\R_+\times\Theta,$
				\begin{numcases}{}
					\partial_t \hat u_t(a,\theta)+\partial_a \hat u_t(a,\theta)=-\hF(t)\gamma(a,\theta)u_t(a,\theta)-\fF(t)\gamma(a,\theta)\hat u_t(a,\theta)-\sqrt{\fF(t)\gamma(a,\theta)u_t(a,\theta)}\zeta_t(a,\theta)\nonumber\\[0.3cm]
					\hat u(t,0,\theta)=\int_{\R_+\times\Theta}\left[\left(\hF(t)u_t(a,\widetilde\theta)+\fF(t)\hat{u}_t(a,\widetilde\theta)\right)\gamma(a,\widetilde\theta)\right.\non\\\hspace{7cm}\left.+\sqrt{\fF(t)\gamma(a,\wt\theta)u_t(a,\wt\theta)}\zeta_t(a,\wt\theta)\right]K(\widetilde\theta,\theta)\rd a\dP(\rd\widetilde{\theta})\label{CLT-pde}\\[0.3cm]
					\hat{u}(0,a,\theta)=\hat{u}_0(a,\theta)\nonumber\\[0.3cm]
					\hF(t)=\displaystyle{\int_{\R_+\times\Theta}\lambda(a,\theta)\hat{u}_t( a,\theta)\rd a\dP(\rd\theta),}\non
				\end{numcases}
				where $\zeta_t$ is a Gaussian space-time white noise (see \cite{nualart2006malliavin} for its properties).  
				Note that $\zeta_t\in\mathcal{W}_0^{-(m_d+1),\alpha}$ and for every $\varphi\in\mathcal{W}_0^{m_d+1,\alpha},$
				\begin{equation*}W_t(\varphi)\stackrel{(d)}{=}\int_{0}^{t}\sqrt{\langle \mu_s,\lambda\rangle}\int_{\R_+\times\Theta} \zeta_s(a,\theta)q_s\varphi(a,\theta)\rd a\dP(\rd\theta) ds,\end{equation*}
				where 
				\[q_s\varphi(a,\theta)=\sqrt{\gamma(a,\theta)u_s(a,\theta)}\int_{\Theta}\left(\varphi(0,\wt\theta)-\varphi(a,\theta)\right)K(\theta,\wt\theta)\nu(d\wt\theta).\]
				The following Proposition characterizes the solution of signed-measure $\hat{\mu}$ when it has a density. 
				 \begin{prop}\label{SPDE-sec-prop}Given $\hat{u}$ solution of the system of equation  \eqref{CLT-pde}, for all $\varphi\in\mathcal{W}^{m_d+2,\alpha}_0,$
				 	\begin{equation}\label{eq-hu}
				 		\langle \hat{u}_t,\varphi\rangle=\langle \hat{u}_0,\varphi\rangle+\int_{0}^{t}\langle \hat{u}_t,\partial_a\varphi\rangle ds+\int_{0}^{t}\left(\langle \hat{u}_s,\lambda\rangle\langle u_s,R\varphi\rangle+\langle u_s,\lambda\rangle \langle \hat{u}_s,R\varphi\rangle\right)ds+W_t(\varphi).
				 	\end{equation}
				 	Conversely given $\hat{\mu}_t(da,d\theta)=\hat{u}_t(a,\theta)da\nu(d\theta),$ solution of equation~\eqref{CLT-exp-VVM}, $\hat{u}$ is the unique solution in law of equation~\eqref{CLT-pde}. 
				 \end{prop}
				 Note that when $\hat{\mu}_t(da,d\theta)=\hat{u}_t(a,\theta)da\nu(d\theta),$ is a solution of equation~\eqref{CLT-exp-VVM}, $\hat{u}$ satisfies equation \eqref{eq-hu}, and we easily derive that $\hat{u}$ is a solution of equation~\eqref{CLT-pde}.
				 The sufficient condition  follows by integration.
				\subsubsection{Link with literature}
				In this section we show that the FCLT describes in Theorem~\ref{CLT-formulation-VVM} corresponds to the result obtained in \cite{ngoufack2025functional}. We recall that when there is no memory ($K(\cdot,\tilde{\theta})=K(\tilde{\theta}$)), the model describes here corresponds to the model introduced in \cite{forien-Zotsa2022stochastic}.  
				
				 We set
				 \[\hS(t,\theta)=\int_{\R_+\times\Theta}\gamma(a,\wt\theta)K(\widetilde\theta,\theta)\hat{u}_t(a,\wt\theta)\rd a\dP(\rd\widetilde{\theta}),\]
				 the fluctuation of the average susceptibility of the population.
				 
				We introduce the following Gaussian process. 
				\begin{definition} We set $(\mathfrak{M}_{0,1},\mathfrak{M}_{0,2},\mathfrak{M}_1,\mathfrak{M}_2)$ a centered Gaussian process with covariance function: for $t,t'\geq0$ and $\varphi,\psi$ two non-negative measurable functions by,
					\begin{enumerate}
						\item
						\begin{multline*}
						\Cov\left(\mathfrak{M}_{0,1}(\varphi)(t),\mathfrak{M}_{0,1}(\psi)(t')\right)=\\\Cov\left(\varphi(a_0+t,\theta_0)\exp\left(-\int_{0}^{t}\fF(s)\gamma(a_0+s,\theta_0)\rd s\right),\psi(a_0+t',\theta_0)\exp\left(-\int_{0}^{t'}\fF(s)\gamma(a_0+s,\theta_0)\rd s\right)\right)
						\end{multline*}
						
						\item 
						\begin{multline*}
							\Cov\left(\mathfrak{M}_{0,2}(\varphi)(t),\mathfrak{M}_{0,2}(\psi)(t')\right)\\
							=\int_{\R_+\times\Theta}\varphi(a+t,\theta)\psi(a+t',\theta) u_0(a,\theta)\exp\left(-\int_{0}^{t\vee t'}\fF(r)\gamma(r+a,\theta)\rd r\right)\rd a\dP(d\theta)\\-\int_{\R_+\times\Theta}\varphi(a+t,\theta)\psi(a+t',\theta)u_0(a,\theta)\exp\left(-\int_{0}^{t}\fF(r)\gamma(r+a,\theta)\rd r-\int_{0}^{t'}\fF(r)\gamma(r+a,\theta)\rd r\right)\rd a\dP(d\theta).
						\end{multline*}
						
						\item 
						\begin{multline*}\Cov\left(\mathfrak{M}_{1}(\varphi)(t),\mathfrak{M}_{1}(\psi)(t')\right)=\\\int_{\Theta}\int_{0}^{t\wedge t'}\varphi(t-s,\theta)\psi(t'-s,\theta)\exp\left(-\int_{s}^{t}\fF(r)\gamma(r-s,\theta)\rd r-\int_{s}^{t'}\fF(r)\gamma(r-s,\theta)\rd r\right)\fS(s,\theta)\fF(s)ds\dP(d\theta).\end{multline*}
						\item 
						\begin{multline*}
							\Cov\left(\mathfrak{M}_{2}(\varphi)(t),\mathfrak{M}_{2}(\psi)(t')\right)\\
							=\int_{\Theta}\int_{0}^{t\wedge t'}\varphi(t-a,\theta)\psi(t'-a,\theta) \exp\left(-\int_{a}^{t\vee t'}\fF(r)\gamma(r-a,\theta)\rd r\right)\fF(a)\fS(a,\theta)\rd a \dP(d\theta)\\-\int_{\Theta}\int_{0}^{t\wedge t'}\varphi(t-a,\theta)\psi(t'-a,\theta)\exp\left(-\int_{a}^{t}\fF(r)\gamma(r-a,\theta)\rd r-\int_{s}^{t'}\fF(r)\gamma(r-a,\theta)\rd r\right)\fF(a)\fS(a,\theta)\rd a\dP(d\theta).
						\end{multline*}
						\item
						\begin{multline*}
							\Cov(\mathfrak{M}_{1}(\varphi)(t),\mathfrak{M}_{2}(\psi)(t'))=\int_{\Theta}\int_{\Theta}\int_{0}^{t\wedge t'}\int_0^{s} \varphi(t-s,\theta)\psi(t'-a,\wt\theta)\gamma(s-a,\wt\theta) \fF(s)\fF(a)\fS(a,\wt\theta)
							\\\exp\left(-\int_{s}^{t}\fF(r)\gamma(r-s,\theta)\rd r-\int_{a}^{t'}\fF(r)\gamma(r-a,\wt\theta)\rd r\right)K(\wt\theta,\theta)\rd a \rd s\dP(\rd\wt\theta)\dP(\rd\theta).
						\end{multline*}
						\item $\Cov(\mathfrak{M}_{1}(\varphi)(t),\mathfrak{M}_{0,1}(\psi)(t'))=\Cov(\mathfrak{M}_{1}(\varphi)(t),\mathfrak{M}_{0,2}(\psi)(t'))=\Cov(\mathfrak{M}_{0,1}(\varphi)(t),\mathfrak{M}_{0,2}(\psi)(t'))=0,$ 
					$\Cov(\mathfrak{M}_{2}(\varphi)(t),\mathfrak{M}_{0,1}(\psi)(t'))=\Cov(\mathfrak{M}_{2}(\varphi)(t),\mathfrak{M}_{0,2}(\psi)(t'))=0.$
					\end{enumerate}
				\end{definition}
				 The expressions of $(\mathfrak{M}_{0,1},\mathfrak{M}_{0,2},\mathfrak{M}_1,\mathfrak{M}_2)$ are given in Appendix Subsection~\ref{A-sec-N}.
				 
				Using the method of characteristics, we easily derive the following expression for the solution. This proposition shows that Theorem~\ref{CLT-formulation-VVM} is equivalent to establishing the fluctuations for the average total force of infection and the average susceptibility of the population.
				
				\begin{prop}\label{eq-equiv-spde}
					Given $\hat{u}$ solution of the system of equation  \eqref{CLT-pde}, the pair $(\hF,\hS)$ given by \eqref{eq:def-F} and \eqref{eq:def-S} respectively, is the unique solution of the following system of integral equations  
						\begin{align}
						\hF(t)&=-\int_0^t\int_{\R_+\times\Theta}\lambda(a+t,\theta) \gamma(a+s,\theta)u_0(a,\theta)\exp\left(-\int_{0}^{t}\fF(r)\gamma(r+a,\theta)\rd r\right) da\dP(d\theta)\hF(s)ds\nonumber\\
						&-\int_{\Theta}\int_{0}^t\int_a^{t}\lambda(t-a,\theta) \gamma(s-a,\theta)\exp\left(-\int_{a}^{t}\fF(r)\gamma(r-a,\theta)\rd r\right)\hF(s)\fF(a)\fS(a,\theta)ds da\dP(d\theta)\nonumber\\
						&\hspace{1cm}+\int_{\Theta}\int_{0}^t\lambda(t-s,\theta)\exp\left(-\int_{s}^{t}\fF(r)\gamma(r-s,\theta)\rd r\right)\left(\hF(s)\fS(s,\theta)+\fF(s)\hS(s,\theta)\right)ds\dP(d\theta)\nonumber
						\\&\hspace{3cm}+\mathfrak{M}_{0,1}(\lambda)(t)-\mathfrak{M}_{0,2}(\lambda)(t)+\mathfrak{M}_{1}(\lambda)(t)-\mathfrak{M}_{2}(\lambda)(t)\label{VVM-CLT-F}
					\end{align}
					and
					\begin{align}
						\hS(t,\theta)&=-\int_0^t\int_{\R_+\times\Theta}\gamma(a+t,\wt\theta) \gamma(a+s,\wt\theta)u_0(a,\wt\theta)\exp\left(-\int_{0}^{t}\fF(r)\gamma(r+a,\wt\theta)\rd r\right) K(\wt\theta,\theta)da\dP(d\wt\theta)\hF(s)ds\nonumber\\
						&-\int_{\Theta}\int_{0}^t\int_a^{t}\gamma(t-a,\wt\theta) \gamma(s-a,\wt\theta)\exp\left(-\int_{a}^{t}\fF(r)\gamma(r-a,\wt\theta)\rd r\right)\hF(s)\fF(a)\fS(a,\wt\theta)K(\wt\theta,\theta)ds da\dP(d\wt\theta)\nonumber\\
						&\hspace{0.5cm}+\int_{\Theta}\int_{0}^t\gamma(t-s,\wt\theta)\exp\left(-\int_{s}^{t}\fF(r)\gamma(r-s,\wt\theta)\rd r\right)\left(\hF(s)\fS(s,\wt\theta)+\fF(s)\hS(s,\wt\theta)\right)K(\wt\theta,\theta)ds\dP(d\wt\theta)\nonumber
						\\&\hspace{3cm}+\mathfrak{M}_{0,1}(\gamma K(\cdot,\theta))(t)-\mathfrak{M}_{0,2}(\gamma K(\cdot,\theta))(t)+\mathfrak{M}_{1}(\gamma K(\cdot,\theta))(t)-\mathfrak{M}_{2}(\gamma K(\cdot,\theta))(t).\label{VVM-CLT-G}
					\end{align}
					Conversely, given $(\hF,\hS)$ solution of the system of equations \eqref{VVM-CLT-F}-\eqref{VVM-CLT-G},
					 	\begin{equation}\label{VVL-eq-CLT-2'}
					 	\hat{u}_t(a,\theta)=\begin{cases}
					 		\hat{u}_0(a-t,\theta)\exp\left(-\int_{0}^{t}\fF(s)\gamma(a-t+s,\theta)\rd s\right)\\\\\hspace{1cm}-\int_0^t F_1(s,a-t+s,\theta)\exp\left(-\int_{s}^{t}\fF(r)\gamma(r+a-t,\theta)\rd r\right)ds&\mbox{ if }a> t,\\\\
					 		\hat{u}(t-a,0,\theta)\exp\left(-\int_{t-a}^{t}\fF(s)\gamma(s-t+a,\theta)\rd s\right)\\\\\hspace{1cm}-\int_0^a F_1(s+t-a,s,\theta)\exp\left(-\int_{t+s-a}^{t}\fF(r)\gamma(r+a-t,\theta)\rd r\right)ds&\mbox{ if }a\leq t,
					 	\end{cases}
					 \end{equation}
					 where \[F_1(t,a,\theta)=\hF(t)\gamma(a,\theta)u_t(a,\theta)+\sqrt{\fF(t)\gamma(a,\theta)u_t(a,\theta)}\zeta_t(a,\theta),\]
					 is the unique solution in law of equation~\eqref{CLT-pde}.
				\end{prop}
				The set of equation \eqref{VVM-CLT-F}-\eqref{VVM-CLT-G} is exactly the set of equation $(3.16)-(3.17)$ in \cite[Proposition~$3.10$]{ngoufack2025functional} when $K(\theta,\wt\theta)=K(\widetilde{\theta})$.
		\section{Proof of Central Limit Theorem}\label{CLT-proof}
			
			We easily check that for any test functions $f:\R_+\times\R_+\times\dR^d\to\R_+$, such that for all $\theta\in\Theta,\, (t,a)\mapsto f_t(a,\theta)=f(t,a,\theta)$ is a continuously differentiable function with respect to its first two variables, we have
			\begin{align}
				\langle\mu ^N_t,f_t\rangle	&=\langle\mu ^N_0,f_0\rangle+\int_0^t \langle\mu ^N_s,\partial_af_s+\partial_sf_s\rangle  \rd s +\int_{0}^{t}\langle\mu ^N_{s},\lambda\rangle\langle\mu ^N_{s},Rf_{s}\rangle\rd s
				\nonumber\\
				&+\frac{1}{N}\sum_{k=1}^N\int_{0}^{t}\int_{\Theta}\int_{0}^{\infty}\left(f_{s}(0,\widetilde\theta)-f_s(a^N_k(s^-),\theta^N_k(s^-))\right)\mathds{1}_{\fF^N(s^-)\gamma^N_k(s^-)K(\theta^N_k(s^-),\widetilde\theta)\geq z}\overline{Q}_k(\rd z,\rd\widetilde \theta,\rd s),
			\end{align}
			where $\overline{Q}_k$ is the compensated Poisson measure of $Q_k$, and $Rf$ is defined by \eqref{eq:Rjump-term}.
			Consequently $\hat{\mu}^N$ satisfies the following equation,
			\begin{equation} \label{CLT-expr-C}
				\langle \hat{\mu}_t^N,f_t\rangle=\langle \hat{\mu}_0^N,f_0\rangle+\int_{0}^{t}\langle \hat{\mu}_s^N,L(f_s)\rangle ds+\int_{0}^{t}V^N_s(f_s)ds + W^N_t(f_t),
			\end{equation} 
			where $L(f_t)(a,\theta)=\partial_t f_t(a,\theta)+\partial_af_t(a,\theta)$, and
			$V^N_t(f_t)=\langle\hat{\mu}^N_t,\lambda\rangle\langle\mu^N_t, R f_t\rangle+\langle\mu_t,\lambda\rangle\langle\hat\mu^N_t, R f_t\rangle$, with
			\[R f_t(a,\theta)=\int_{\Theta}\left(f_t(0,\wt\theta)-f_t(a,\theta)\right)\gamma(a,\theta)K(\theta,\wt\theta)\nu(d\wt\theta)\]
			and 
			\begin{equation}\label{Mart-W}
			W^N_t(f_t)=\frac{1}{\sqrt{N}}\sum_{k=1}^{N}\int_0^t\int_{\Theta}\int_{0}^{+\infty}\left(f_s(0,\wt\theta)-f_s(a^N_k(s^-),\theta^N_k(s^-))\right)\indic{\fF^N(s^-)\gamma_k^N(s^-)K(\theta^N_k(s^-),\wt\theta)\geq z}\overline{Q}_k(ds,d\wt\theta,dz)\end{equation}
			The quadratic variation of $W^N_t(f_t)$ is given by
			\begin{equation}\ll W^N\gg_t(f_t)=\int_{0}^{t}\langle \mu_s^N,\lambda\rangle\langle \mu_s^N,R^{(2)}f_s\rangle ds\label{CLT-eq-W_q}\end{equation}
			where
			\[R^{(2)}f_t(a,\theta)=\int_{\Theta}\left(f_t(0,\wt\theta)-f_t(a,\theta)\right)^2\gamma(a,\theta)K(\theta,\wt\theta)\nu(d\wt\theta).\]

%
		To prove the convergence of Theorem~\ref{CLT-formulation-VVM}, we first establish tightness of $\hat{\mu}^N$ in a suitable Sobolev space, which implies the existence of a convergent subsequence of $\hat{\mu}^N.$ Then we show that the limit of the subsequence does not depend on the choice of the subsequence and conclude using continuous mapping theorem.
			
			\subsection{Tightness}\label{tight-secsub}
			In this section, we will establish tightness of the process $\hat{\mu}^N$. 
			To do this, we will use the Aldous tightness criterion for Hilbert space-valued stochastic process, as introduced in \cite[pages~$34-35$]{joffe1986weak} (see Definition~\ref{Def-tight} in Appendix). We also refer to \cite[page~$13$]{chevallier_fluctuations_2017}. According to Aldous criterion, we need some moment inequalities to prove tightness. Therefore, we will first establish the following preliminary results. 
			\subsubsection{Preliminary results}
			We first introduce the following two process.
			For $1\leq k\leq N$ and $t\geq 0$, let $A^N_k(t)$ be the number of times that the $k$-th individual with trait $\theta^N_k$ has been (re-)infected on the time interval $(0,t]$. The process $A^N_k$ is then define as follows.
			\[A^N_k(t)=\int_{0}^{t}\int_{\Theta}\int_{0}^{\infty}\indic{\fF^N(s^-)\gamma^N_k(s^-)K(\theta^N_k(s^-),\widetilde\theta)\geq z}Q_k(\rd z,\rd \widetilde\theta,\rd s).\] 
			We consider an i.i.d copy $\left((a_k(t),\theta_k(t))\right)_{k\geq1}$ of the pair $(a(t),\theta(t))$ given by Remark~\ref{Rq-a-theta} using the same $(Q_k)_k$ as in \eqref{eq:SDE}.
			We then introduce for every $k\geq 1$ and $t\geq 0$, the following process \[A_k(t)=\int_{0}^{t}\int_{\Theta}\int_{0}^{\infty}\indic{\fF(s^-)\gamma_k(s^-)K(\theta_k(s^-),\widetilde\theta)\geq z}Q_k(\rd z,\rd \widetilde\theta,\rd s).\]
			In the expression of $A_k$, we use the same $(\lambda,\gamma)$ and $(Q_k)_k$ as in \eqref{eq:SDE}. As the family $(a_k,\theta_k)_{k\geq0}$ are i.i.d, $(A_k)_{k\geq1}$ are also i.i.d. 
			
			 
			 Now for each $k\geq1,$ we compare the process $A^N_k(t)$ with the process $A_k(t)$. 
			 
			\begin{lemma}\label{VVM-CLT-lem_inq}Under Assumption~\ref{VVM-CLT-ass-a_0}, for $k\in\mathbb{N}$ and $T\geq0$, 
				\begin{multline}
					\mathbb{E}\left[\sup_{t\in[0,T]}\left|A^N_k(t)-A_k(t)\right|\Big|\left(a_0^{k'},\theta_0^{k'}\right)_{1\leq k'\leq k}\right]\\
					\begin{aligned}
					&\leq\int_{0}^{T}\int_{\Theta}\mathbb{E}\Big[\left|\fF^N(s^-)\gamma^N_k(s^-)K(\theta^N_k(s^-),\widetilde\theta)-\fF(s^-)\gamma_k(s^-)K(\theta_k(s^-),\widetilde\theta)\right|\Big|\left(a_0^{k'},\theta_0^{k'}\right)_{1\leq k'\leq k}\Big]\dP(\rd\widetilde\theta)ds\nonumber\\&=:\delta^N(T)\label{FLLN-eqA}
					\end{aligned}
				\end{multline}
				and 
				\begin{equation*}
					\mathbb{E}\left[\sup_{t\in[0,T]}\left|a^N_k(t)-a_k(t)\right|\Big|\left(a_0^{k'},\theta_0^{k'}\right)_{1\leq k'\leq k}\right]\leq T\delta^N(T).
				\end{equation*}
				Moreover, 
				\begin{equation}\delta^N(T)\leq\frac{\lambda^*}{\sqrt{N}}T\exp\left(4T\lambda^*\int_\Theta\sup_{\theta\in\Theta}K(\theta,\widetilde\theta)\dP(\rd\widetilde\theta)\right).\label{VVM-FLLN-eqdelta}\end{equation}
			\end{lemma}
			\begin{proof}
				The proof follows exactly as in the proof of \cite[Lemma~$6.2$]{forien-Zotsa2022stochastic}, taking, 
				\[\lambda_{k,A^N_k(t)}(a^N_k(t)):=\lambda(a^N_k(t),\theta^N_k(t),\text{ and }\gamma_{k,A^N_k(t)}(a^N_k(t)):=\gamma(a^N_k(t),\theta^N_k(t).\]
				In \cite{forien-Zotsa2022stochastic} the elapsed time since the last infection corresponds to what we define as age of infection. 
				Similarly, we take $\left(\lambda_{k,A_k(t)}(a_k(t)),\gamma_{k,A_k(t)}(a_k(t))\right)$ in the same manner.

				The main difference is that, in \cite{forien-Zotsa2022stochastic} the rate at which the process $A_k^N$ jumps, $\Upsilon^N_k$ and the rate at which the process $A_k$ jumps, $\Upsilon_k,$ do not depend on the parameter $\theta$. The only difference with the proof in \cite[Lemma~$6.2$]{forien-Zotsa2022stochastic} arises from the fact that we use the following inequality $$\int_{\Theta}\left|\fF(s)\gamma_1(s)K(\theta_1(s),\widetilde\theta)-\fF^N(s)\gamma^N_1(s)K(\theta^N_1(s),\widetilde\theta)\right|\dP(\rd\wt\theta)\leq 2\lambda_*\int_\Theta\sup_{\theta\in\Theta}K(\theta,\widetilde\theta)\dP(\rd\widetilde\theta)$$ instead of the inequality $|\Upsilon_1(s)-\Upsilon_1^N(s)|\leq\lambda_*$ given in \cite[Lemma~$6.2$]{forien-Zotsa2022stochastic}.
			\end{proof}
			As in \cite{ngoufack2025functional} we introduce 
			\[\chi_N^{(k)}(t,a_0,\theta_0)=\mathbb{P}\left( (a^N_{k'}(s))_{s\in[0,t]}\neq (a_{k'}(s))_{s\in[0,t]},\forall k'=1,\cdots,k\big|\left(a_0^{i},\theta_0^{i}\right)_{1\leq i\leq k}\right),\]
			and we deduce the following Proposition for the moments inequality.
			\begin{prop}\label{VVM-CLT-prop-inq}
				Under Assumption~\ref{VVM-CLT-ass-a_0} for all $N\geq k,$ and $t\in[0,T]$ there are positive constant $C_{k,T}$ independent of $(a_0^i)_{1\leq i\leq k}$ such that,
				\[\chi_N^{(k)}(t,a_0,\theta_0)\leq C_{k,T}N^{-k/2},\textbf{\rm and }\E\left[|\fF^N(t)-\fF(t)|^k\big|\left(a_0^{i},\theta_0^{i}\right)_{1\leq i\leq k}\right]\leq C_{k,T}N^{-k/2}.\]
			\end{prop}
			\begin{proof}
				The proof follows exactely the same reasoning as in \cite[Proposition~$5.4$]{ngoufack2025functional}. More precisely, in the proof of \cite[Proposition~$5.4$]{ngoufack2025functional}, we replace $\Delta^N_k(t)$ with the following expression
				\[\Delta^N_k(t):=\int_{0}^{t}\int_{\Theta}\int_{0}^{\infty}\left|\indic{\fF(s^-)\gamma(s^-)K(\theta(s^-),\widetilde\theta)\geq z}-\indic{\fF^N(s^-)\gamma^N_k(s^-)K(\theta^N_k(s^-),\widetilde\theta)\geq z}\right|Q_k(\rd z,\rd \widetilde\theta,\rd s).\] 
				As in the proof of Lemma~\ref{VVM-CLT-lem_inq}, $\lambda_{A^N_k(t)}(a^N_k(t)):=\lambda(a^N_k(t),\theta^N_k(t))$ and $\gamma_{A^N_k(t)}(a^N_k(t)):=\gamma(a^N_k(t),\theta^N_k(t)).$ Similarly $\lambda_{A^N_k(t)}(a^N_k(t)):=\lambda(a^N_k(t),\theta^N_k(t))$ and $\gamma_{A_k(t)}(a_k(t)):=\gamma(a_k(t),\theta_k(t)).$
				
				We recall from \cite{forien-Zotsa2022stochastic} that $\Upsilon^N_k$ is the rate at which process $A^N_k$ jumps, while $\Upsilon_k$ corresponds to the rate at which the process $A_k$ jumps.
				Consequently, instead of using the inequality $|\Upsilon_1(s)-\Upsilon_1^N(s)|\leq\lambda_*,$ given in \cite[Proposition~$5.4$]{ngoufack2025functional}, we use $$\int_{\Theta}\left|\fF(s)\gamma_1(s)K(\theta_1(s),\widetilde\theta)-\fF^N(s)\gamma^N_1(s)K(\theta^N_1(s),\widetilde\theta)\right|\dP(\rd\wt\theta)\leq 2\lambda_*\int_\Theta\sup_{\theta\in\Theta}K(\theta,\widetilde\theta)\dP(\rd\widetilde\theta).$$ 
				Additionally, we also use the fact that
				\begin{multline*}
					\int_{\Theta}\left|\fF(s)\gamma_1(s)K(\theta_1(s),\widetilde\theta)-\fF^N(s)\gamma^N_1(s)K(\theta^N_1(s),\widetilde\theta)\right|\dP(\rd\wt\theta)\\
					\begin{aligned}
						&\leq |\fF^N(s)-\fF(s)|+\lambda_*\int_{\Theta}\left|\gamma_1(s)K(\theta_1(s),\widetilde\theta)-\gamma^N_1(s)K(\theta^N_1(s),\widetilde\theta)\right|\dP(\rd\wt\theta)\\
					&\leq |\fF^N(s)-\fF(s)|+2\lambda_*\int_\Theta\sup_{\theta\in\Theta}K(\theta,\widetilde\theta)\dP(\rd\widetilde\theta)\indic{a^N_1(s)\neq a_1(s)\text{ or }\theta^N_1(s)\neq\theta_1(s)}\\
					&\leq |\fF^N(s)-\fF(s)|+2\lambda_*\int_\Theta\sup_{\theta\in\Theta}K(\theta,\widetilde\theta)\dP(\rd\widetilde\theta)\Delta^N_1(s),
					\end{aligned}
				\end{multline*} 
				instead of the inequality $|\Upsilon_1(s)-\Upsilon_1^N(s)|\leq|\fF^N(s)-\fF(s)|+\lambda_*\Delta^N_1(s)$ given in \cite[Proposition~$5.4$]{ngoufack2025functional}.
			\end{proof} 
			We also establish the following Lemma.
			\begin{lemma}\label{CLT-VVM-lem-inqf}
				For any $\alpha\in\R_+$ and $x:=(a,\theta),y=(a',\theta')\in\R_+\times\Theta,$ the mapping $\delta_x$ and $D_{x,y}:\mathcal{W}_0^{m_d,\alpha}\to\R,$ defined by $\delta_x(\varphi)=\varphi(x)$ and $D_{x,y}(\varphi)=\varphi(x)-\varphi(y)$ are linear continuous. In particular there exists a deterministic constant $C>0$ independent of $\theta$ and $\theta',$ 
				\begin{equation}
					\left\{
					\begin{aligned}
						&\|\delta_x\|_{-(m_d+1),\alpha}\leq\|\delta_x\|_{-m_d,\alpha}\leq C(1+a^\alpha)\\
						&\|D_{x,y}\|_{-(m_d+1),\alpha}\leq\|D_{x,y}\|_{-m_d,\alpha}\leq C(1+(a\vee a')^\alpha).
					\end{aligned}\right.
				\end{equation}
			\end{lemma}
			\begin{proof}
				For all $\varphi,$ we have $|D_{x,y}(\varphi)|\leq|\delta_x(\varphi)|+|\delta_y(\varphi)|$. Hence it suffices to show that $\|\delta_x\|_{-m_d,\alpha}\leq C_1(1+a^\alpha)$. We have for  $x=(a,\theta)$,
				\[|\delta_x(\varphi)|=|\varphi(a,\theta)|\leq\|\varphi\|_{\mathcal{C}^{0,\alpha}}(1+a^{\alpha})\leq  C\|\varphi\|_{m_d,\alpha}(1+a^\alpha),\]
				where we use the fact $ \mathcal{W}_0^{m_d,\alpha}\hookrightarrow\mathcal{C}^{0,\alpha}.$
			\end{proof}
		\subsubsection{Suitable Sobolev space to establish the tightness of $\hat{\mu}^N$ }
			\begin{lemma}\label{CLT-lem-eta}
				Under Assumption~\ref{VVM-CLT-ass-a_0}, for all $T\geq 0,$
				\[\sup_{N\geq 1}\sup_{0\leq t\leq T}\E\left[\|\hat{\mu}^N_t\|_{-m_d,\alpha}^2\right]<\infty.\]
			\end{lemma}
			\begin{proof}We adapt the proof of \cite[Proposition~$4.7$]{chevallier_fluctuations_2017}.
				Let $(\varphi_k)_{k\geq 1}$ be an orthonormal basis of $\mathcal{W}^{m_d,\alpha}_0$. By Parseval's identity,
				\[\|\hat{\mu}^N_t\|_{-m_d,\alpha}^2=\sum_{k=1}^\infty\langle\hat{\mu}^N_t,\varphi_k\rangle^2.\]
				We have 
				\begin{align}\label{VVM-CLT-eq-mu-0}
					\langle\hat{\mu}^N_t,\varphi_k\rangle&=\sqrt{N}\left(\frac{1}{N}\sum_{j=1}^{N}\varphi_k(a_j^N(t),\theta_j^N(t))-\E\left[\varphi_k(a_1(t),\theta_1(t))\right]\right)\nonumber\\
					&:=\Xi^N_{k,1}(t)+\Xi^N_{k,2}(t),
				\end{align}
				where 
				\begin{numcases}{}
					\Xi^N_{k,1}(t)=\sqrt{N}\left(\frac{1}{N}\sum_{j=1}^{N}\left(\varphi_k(a_j^N(t),\theta_j^N(t))-\varphi_k(a_j(t),\theta_j(t))\right)\right)\nonumber\\
					\Xi^N_{k,2}(t)=\sqrt{N}\left(\frac{1}{N}\sum_{j=1}^{N}\varphi_k(a_j(t),\theta_j(t))-\E\left[\varphi_k(a_1(t),\theta_1(t))\right]\right)\nonumber
				\end{numcases}
				Using the independence of the family $(a_j,\theta_j)_j,$ it follows that,
				\begin{align*}
					\E\left[\sum_{k\geq 1}\Xi^N_{k,2}(t)^2\right]&=\frac{1}{N}\sum_{k\geq 1}\sum_{j=1}^{N}\E\left[\left(\varphi_k(a_j(t),\theta_j(t))-\E\left[\varphi_k(a_1(t),\theta_1(t))\right]\right)^2\right]\nonumber\\
					&=\sum_{k\geq 1}\E\left[\left(\varphi_k(a_1(t),\theta_1(t))-\E\left[\varphi_k(a_1(t),\theta_1(t))\right]\right)^2\right]\nonumber\\
					&\leq\sum_{k\geq 1}\E\left[\left(\varphi_k(a_1(t),\theta_1(t))\right)^2\right]\nonumber\\
					&=\sum_{k\geq 1}\E\left[\langle\delta_{(a_1(t),\theta_1(t))},\varphi_k\rangle^2\right]\nonumber\\
					&=\E\left[\|\delta_{(a_1(t),\theta_1(t))}\|_{-m_d,\alpha}^2\right]\nonumber\\
					&\leq C^2\E\left[\left(1+(a_0^1+T)^\alpha\right)^2\right]\\
					&\leq2^{2\alpha}C^2(1+\E[a_0^{2\alpha}]+T^{2\alpha})<\infty,
				\end{align*} 
				where line six follows from Lemma~\ref{CLT-VVM-lem-inqf} and the fact that $a_1(t)\leq a_0^1+T,$ while the last line follows from convexity and Assumption~\ref{VVM-CLT-ass-a_0}.
				
				Consequently 
				\begin{equation}\label{VVM-CLT-eq-def-0}
					\sup_{N\geq 1}\sup_{0\leq t\leq T} \E\left[\sum_{k\geq 1}\Xi^N_{k,2}(t)^2\right]<\infty.
				\end{equation}
				On the other hand, by exchangeability we have 
				\begin{multline}\label{VVM-CLT-eq-def-1'}
					\E\left[\sum_{k\geq 1}\left(\Xi^N_{k,1}(t)\right)^2\right]=\frac{1}{N}\sum_{k\geq 1}\E \left[\left(\sum_{j=1}^{N}\left(\varphi_k(a_j^N(t),\theta_j^N(t))-\varphi_k(a_j(t),\theta_j(t))\right)\right)^2\right]\\
					\begin{aligned}
						&=\sum_{k\geq 1}\E \left[\left(\varphi_k(a_1^N(t),\theta_1^N(t))-\varphi_k(a_1(t),\theta_1(t))\right)^2\right]\\
						&\hspace{1cm}+(N-1)\sum_{k\geq 1}\E \left[\left(\varphi_k(a_1^N(t),\theta_1^N(t))-\varphi_k(a_1(t),\theta_1(t))\right)\left(\varphi_k(a_2^N(t),\theta_2^N(t))-\varphi_k(a_2(t),\theta_2(t))\right)\right].
					\end{aligned}
				\end{multline}
				In what follows, we use the notation $D$ introduced in Lemma~\ref{CLT-VVM-lem-inqf}.
				
				Since $a_i^N(t),\,a_i(t)$ are upper bounded by $a_0^i+t$ and $\theta_i^N(t), \theta_i(t)\in\Theta$,
				\begin{multline}\label{VVM-CLT-eq-def-1}
					(N-1)\sum_{k\geq 1}\E \left[\left(\varphi_k(a_1^N(t),\theta_1^N(t))-\varphi_k(a_1(t),\theta_1(t))\right)\left(\varphi_k(a_2^N(t),\theta_2^N(t))-\varphi_k(a_2(t),\theta_2(t))\right)\right]\\
					\begin{aligned}
						&\leq (N-1)\E\left[\indic{\left\{(a^N_{k'}(t)\neq (a_{k'}(t)),\,k'=1,2\right\}}\sup_{x,y\leq a_0^1\vee a_0^2+T,\theta,\theta'\in\Theta}\sum_{k\geq 1}\left|\varphi_k(x,\theta)-\varphi_k(y,\theta')\right|^2\right]\\
						&=(N-1)\E\left[\indic{\left\{(a^N_{k'}(t)\neq (a_{k'}(t)),\,k'=1,2\right\}}\sup_{x,y\leq a_0^1\vee a_0^2+T,\theta,\theta'\in\Theta}\left\|D_{(x,\theta),(y,\theta')}\right\|_{-m_d,\alpha}^2\right]\\
						&\leq2^{2\alpha} C^2(N-1)\E\left[\indic{\left\{(a^N_{k'}(t)\neq (a_{k'}(t)),\,k'=1,2\right\}}\left(1+(a_0^1\vee a^2_0)^{2\alpha}+T^{2\alpha}\right)\right]\\
						&=2^{2\alpha} C^2(N-1)\E\left[\chi_N^{(2)}(t,a_0,\theta_0)\left(1+(a_0^1\vee a^2_0)^{2\alpha}+T^{2\alpha}\right)\right]\\
						&\leq 2^{4\alpha-1} C^2C_{2,T}(1+2\E\left[a_0^{2\alpha}\right]+T^{2\alpha})<\infty,
					\end{aligned}
				\end{multline}
				where the second equality follows from the fact that $(\varphi_k)_k$ is an orthogonal basis of $\mathcal{W}_0^{m_d,\alpha},$ and while the next equality from Lemma~\ref{CLT-VVM-lem-inqf}. The last line follows from Proposition~\ref{VVM-CLT-prop-inq}, convexity and Assumption~\ref{VVM-CLT-ass-a_0}.

				Using again the fact that $(\varphi_k)_k$ is an orthonormal basis of $\mathcal{W}_0^{m_d,\alpha}$, Lemma~\ref{CLT-VVM-lem-inqf} and Assumption~\ref{VVM-CLT-ass-a_0}, we have
				\begin{align}
					\sum_{k\geq 1}\E \left[\left(\varphi_k(a_1^N(t),\theta_1^N(t))-\varphi_k(a_1(t),\theta_1(t))\right)^2\right]&=\E\left[\sum_{k\geq 1}\langle D_{(a_1^N(t),\theta_1^N(t)),(a_1(t),\theta_1(t))},\varphi_k\rangle^2\right]\nonumber\\
					&=\E\left[\|D_{(a_1^N(t),\theta_1^N(t)),(a_1(t),\theta_1(t))}\|_{-m_d,\alpha}^2\right]\nonumber\\
					&\leq 2^{4\alpha-1}C^2\left(1+\E[a_0^{2\alpha}]+T^{2\alpha}\right)<\infty.\label{VVM-CLT-eq-def-2}
				\end{align}
				Consequently from \eqref{VVM-CLT-eq-def-1'} ,\eqref{VVM-CLT-eq-def-1} and \eqref{VVM-CLT-eq-def-2} it follows that,
				\begin{equation}\label{VVM-CLT-eq-def-3}
					\sup_{N\geq 1}\sup_{0\leq t\leq T} \E\left[\sum_{k\geq 1}\Xi^N_{k,1}(t)^2\right]<\infty.
				\end{equation}
				As a result from \eqref{VVM-CLT-eq-mu-0}, \eqref{VVM-CLT-eq-def-0} and \eqref{VVM-CLT-eq-def-3} it follows that,
				\begin{align*}
					\sup_{N\geq 1}\sup_{0\leq t\leq T} \E\left[\|\hat{\mu}^N\|_{-m_d,\alpha}^2\right]\leq 2\left(\sup_{N\geq 1}\sup_{0\leq t\leq T} \E\left[\sum_{k\geq 1}\Xi^N_{k,1}(t)^2\right]+\sup_{N\geq 1}\sup_{0\leq t\leq T} \E\left[\sum_{k\geq 1}\Xi^N_{k,2}(t)^2\right]\right)<\infty.
				\end{align*}
			\end{proof}
			Lemma~\ref{CLT-lem-eta} implies that for $\alpha>\frac{1}{2},$ we have $\hat{\mu}^N\in \mathcal{W}_0^{-m_d,\alpha}$. We will use this Sobolev space in what follows to  establish tightness.
			\subsubsection{Proof of tightness}
			We will first shown that the  moments of the supremum of the process $(W^N)$ and $(V^N)$ are well defined in suitable Sobolev spaces. Then we will deduce their tightness as well $(\hat{\mu}^N)$. 
			\begin{lemma}\label{CLT-mart-W}
				Under Assumption~\ref{V-As-1}, Assumption~\ref{V-As-2} and Assumption~\ref{VVM-CLT-ass-a_0}, $(W^N_t)_t$ is a martingale with paths in $D(\R_+,\mathcal{W}_0^{-m_d,\alpha})$ almost surely. Moreover, for all $T\geq 0,$
				\begin{equation}\sup_{N}\E\left[\sup_{0\leq t\leq T}\|W^N_{t}\|_{-m_d,\alpha}^2\right]<\infty.\label{CLT-VVM-in-M}\end{equation}
			\end{lemma}
			\begin{proof}We adapt the proof of \cite[Proposition~$4.7$]{chevallier_fluctuations_2017}. We recall \eqref{Mart-W}
				\begin{equation*}
					W^N_t(f_t)=\frac{1}{\sqrt{N}}\sum_{k=1}^{N}\int_0^t\int_{\Theta}\int_{0}^{+\infty}\left(f_s(0,\wt\theta)-f_s(a^N_k(s^-),\theta^N_k(s^-))\right)\indic{\fF^N(s^-)\gamma_k^N(s^-)K(\theta^N_k(s^-),\wt\theta)\geq z}\overline{Q}_k(ds,d\wt\theta,dz)\end{equation*}
				 We start to proof \eqref{CLT-VVM-in-M}.
				Let $(\varphi_k)_{k\geq 1}$ be an orthonormal basis of $\mathcal{W}^{m_d,\alpha}_0$. By Parseval's identity
				\begin{equation*}
					\|W^N_t\|_{-m_d,\alpha}^2=\sum_{k\geq 1}\left(W^N_t(\varphi_k)\right)^2.
				\end{equation*}
				Using the fact that for all $k\geq 1,\, (W^N_t(\varphi_k))_t$ is a martingale, from Doob's Maximal inequality, orthogonality of $(\overline{Q}_k)_k$ and exchangeability, from Assumption~\ref{V-As-1}, it follows that,
				\begin{multline*}
					\E\left[\sup_{0\leq t\leq T}\|W^N_t\|_{-m_d,\alpha}^2\right]=\sum_{k\geq 1}\E \left[\sup_{0\leq t\leq T}\left(W^N_t(\varphi_k)\right)^2\right]\nonumber\\
					\begin{aligned}
					&\leq 4\sum_{k\geq 1}\E\left[\left(W^N_T(\varphi_k)\right)^2\right]\nonumber\\
					&=4\sum_{k\geq 1}\int_0^{T}\int_{\Theta}\E\left[\left(\varphi_k(0,\wt\theta)-\varphi_k(a_1(s),\theta_1(s))\right)^2\fF^N(s)\gamma^N_1(s)K(\theta^N_1(s),\wt\theta)\right]\dP(\rd\wt\theta)ds\nonumber\\
					&\leq4\lambda_*\int_0^{T}\int_{\Theta}\E\left[\|D_{(0,\wt\theta),(a_1(s),\theta_1(s))}\|_{-m_d,\alpha}^2K(\theta^N_1(s),\wt\theta)\right]\dP(\rd\wt\theta)ds\nonumber\\
					&\leq4C^2\lambda_*\int_0^{T}\int_{\Theta}\E\left[\left(1+(a_0^1+T)^{\alpha}\right)^2K(\theta^N_1(s),\wt\theta)\right]\dP(\rd\wt\theta)ds\nonumber\\
					&\leq 2^{2\alpha+2}C^2\lambda_*T(1+\E[a_0^{2\alpha}]+T^{2\alpha})<\infty,
					\end{aligned}
				\end{multline*}
				where the last two lines follows from Lemma~\ref{CLT-VVM-lem-inqf}, Assumption~\ref{V-As-2} and Assumption~\ref{VVM-CLT-ass-a_0}. This conclude \eqref{CLT-VVM-in-M}.

				The integrability of $W^N_t$ in $\mathcal{W}_0^{-m_d,\alpha}$ follows from \eqref{CLT-VVM-in-M}. Gathering this with the fact that for all $k\geq 1,\, (W^N_t(\varphi_k))_t$ is a martingale, it follows that  $(W^N_t)_t$ is a $\mathcal{W}_0^{-m_d,\alpha}$-valued martingale. To conclude it remains to show that  $(W^N_t)_t$ is cadlag. This follows from inequality \eqref{CLT-VVM-in-M} and the fact that $(W^N_t(\varphi_k))_t$ is cadlag, see proof's of \cite[Proposition~$4.7$]{chevallier_fluctuations_2017} for more details.  
			\end{proof}
			We will use the following Lemma to control the moments of the process $(V^N_t)_t$ in Lemma~\ref{CLT-lem-V}.  
		\begin{lemma}\label{CLT-VVM-lem-vf} Under Assumption~\ref{VVM-CLT-ass-a_0} and let $(\varphi_k)_k$ be an orthonormal basis of $\mathcal{W}^{m_d,\alpha}_0$. Then,
			\begin{equation}\label{CLT-VVM-eqR_phi}
				\sup_{N\geq 1}\sup_{0\leq t\leq T}\E\left[\sum_{k\geq 1}\langle\hat\mu^N_t, R\varphi_k\rangle^2\right]<\infty.
			\end{equation}
		\end{lemma}
		\begin{proof}
			We recall \eqref{eq:Rjump-term},
			\begin{equation*}
				Rf(a,\theta)=\int_{\Theta}\left(f(0,\widetilde\theta)-f(a,\theta)\right)\gamma( a,\theta)K(\theta,\widetilde\theta)\dP(\rd\widetilde\theta).
			\end{equation*}
			
			We have 
			\begin{align}\label{VVM-CLT-eq-mu-0'}
				\langle\hat{\mu}^N_t,R\varphi_k\rangle&=\sqrt{N}\left(\frac{1}{N}\sum_{j=1}^{N}R\varphi_k(a_j^N(t),\theta_j^N(t))-\E\left[R\varphi_k(a_1(t),\theta_1(t))\right]\right)\nonumber\\
				&:=\Xi^N_{k,3}(t)+\Xi^N_{k,4}(t),
			\end{align}
			where 
			\begin{numcases}{}
				\Xi^N_{k,3}(t)=\sqrt{N}\left(\frac{1}{N}\sum_{j=1}^{N}\left(R\varphi_k(a_j^N(t),\theta_j^N(t))-R\varphi_k(a_j(t),\theta_j(t))\right)\right)\nonumber\\
				\Xi^N_{k,4}(t)=\sqrt{N}\left(\frac{1}{N}\sum_{j=1}^{N}R\varphi_k(a_j(t),\theta_j(t))-\E\left[R\varphi_k(a_1(t),\theta_1(t))\right]\right)\nonumber
			\end{numcases}
			We recall
			Using the independence of the family $(a_j,\theta_j)_j,\,\gamma\leq 1$ (Assumption~\ref{V-As-1}) and Fubini Theorem's with respect to measure $K(\cdot,\wt\theta)\dP(\rd\wt\theta)$ (Assumption~\ref{V-As-2}), it follows that,
			\begin{align*}
				\E\left[\sum_{k\geq 1}\Xi^N_{k,4}(t)^2\right]&=\frac{1}{N}\sum_{k\geq 1}\sum_{j=1}^{N}\E\left[\left(R\varphi_k(a_j(t),\theta_j(t))-\E\left[R\varphi_k(a_1(t),\theta_1(t))\right]\right)^2\right]\nonumber\\
				&=\sum_{k\geq 1}\E\left[\left(R\varphi_k(a_1(t),\theta_1(t))-\E\left[R\varphi_k(a_1(t),\theta_1(t))\right]\right)^2\right]\nonumber\\
				&\leq\sum_{k\geq 1}\E\left[\left(R\varphi_k(a_1(t),\theta_1(t))\right)^2\right]\nonumber\\
				&\leq\sum_{k\geq 1}\E\left[\int_{\Theta}\left(\varphi_k(0,\wt\theta)-\varphi_k(a_1(t),\theta_1(t))\right)^2K(\theta_1(t),\wt\theta)\dP(\rd\wt\theta)\right]\nonumber\\
				&=\int_{\Theta}\E\left[\|D_{(0,\wt\theta),(a_1(t),\theta_1(t))}\|_{-m_d,\alpha}^2K(\theta_1(t),\wt\theta)\right]\dP(\rd\wt\theta)\nonumber\\
				&\leq 2^{2\alpha}C^2\left(1+\E[a_0^{2\alpha}]+T^{2\alpha}\right),
			\end{align*} 
			where the last line follows from Lemma~\ref{CLT-VVM-lem-inqf} and Assumption~\ref{V-As-2}.
			
			Consequently from Assumption~\ref{VVM-CLT-ass-a_0} 
			\begin{equation}\label{VVM-CLT-eq-def-0'}
				\sup_{N\geq 1}\sup_{0\leq t\leq T} \E\left[\sum_{k\geq 1}\Xi^N_{k,4}(t)^2\right]<\infty.
			\end{equation}
			On the other hand, by exchangeability we have 
			\begin{multline}\label{VVM-CLT-eq-def-1'''}
				\E\left[\sum_{k\geq 1}\left(\Xi^N_{k,3}(t)\right)^2\right]=\frac{1}{N}\sum_{k\geq 1}\E \left[\left(\sum_{j=1}^{N}\left(R\varphi_k(a_j^N(t),\theta_j^N(t))-R\varphi_k(a_j(t),\theta_j(t))\right)\right)^2\right]\\
				\begin{aligned}
					&=\sum_{k\geq 1}\E \left[\left(R\varphi_k(a_1^N(t),\theta_1^N(t))-R\varphi_k(a_1(t),\theta_1(t))\right)^2\right]\\
					&+(N-1)\sum_{k\geq 1}\E \left[\left(R\varphi_k(a_1^N(t),\theta_1^N(t))-R\varphi_k(a_1(t),\theta_1(t))\right)\left(R\varphi_k(a_2^N(t),\theta_2^N(t))-R\varphi_k(a_2(t),\theta_2(t))\right)\right].
				\end{aligned}
			\end{multline}
			Since $a_i^N(t),\,a_i(t)$ are upper bounded by $a_0^i+t$ and $\theta_i^N(t), \theta_i(t)\in\Theta$ and as in \eqref{VVM-CLT-eq-def-1},
			\begin{multline}\label{VVM-CLT-eq-def-1''}
				(N-1)\sum_{k\geq 1}\E \left[\left(R\varphi_k(a_1^N(t),\theta_1^N(t))-R\varphi_k(a_1(t),\theta_1(t))\right)\left(R\varphi_k(a_2^N(t),\theta_2^N(t))-R\varphi_k(a_2(t),\theta_2(t))\right)\right]\\
				\begin{aligned}
					&\leq (N-1)\E\left[\indic{a^N_i(t)\neq a_i(t),i=1,2}\sup_{x,y\leq a_0^1\vee a_0^2+T,\theta,\theta'\in\Theta}\sum_{k\geq 1}\left|R\varphi_k(x,\theta)-R\varphi_k(y,\theta')\right|^2\right]\\
					&\leq C_{2,T}\E\left[\sup_{x,y\leq a_0^1\vee a_0^2+T,\theta,\theta'\in\Theta}\sum_{k\geq 1}\left|R\varphi_k(x,\theta)-R\varphi_k(y,\theta')\right|^2\right]\\
					&\leq 2C_{2,T}\E\left[\sup_{x\leq a_0^1\vee a_0^2+T,\theta\in\Theta}\sum_{k\geq 1}\left(R\varphi_k(x,\theta)\right)^2\right]\\
					&\leq 2C_{2,T}\E\left[\sup_{x\leq a_0^1\vee a_0^2+T,\theta\in\Theta}\int_{\Theta}\|D_{(0,\wt\theta),(x,\theta)}\|_{-m_d,\alpha}^2K(\theta,\wt\theta)\dP(\rd\wt\theta)\right]\\
					&\leq 2^{4\alpha}C_{2,T}C^2\left(1+2\E[a_0^{2\alpha}]+T^{2\alpha}\right),
				\end{aligned}
			\end{multline}
			where the second inequality follows from Proposition~\ref{VVM-CLT-prop-inq} and the last two lines from the fact that $(\varphi_k)_k$ is an orthonormal basis of $\mathcal{W}_0^{m_d,\alpha}$ and from Lemma~\ref{CLT-VVM-lem-inqf} and Assumption~\ref{V-As-2}.

			As in \eqref{VVM-CLT-eq-def-1''} we have,
			\begin{multline}
				\sum_{k\geq 1}\E \left[\left(R\varphi_k(a_1^N(t),\theta_1^N(t))-R\varphi_k(a_1(t),\theta_1(t))\right)^2\right]\\
				\begin{aligned}
				&=\E\left[\indic{a^N_1(t)\neq a_1(t)}\sup_{x,y\leq a_0^1+T,\theta,\theta'\in\Theta}\sum_{k\geq 1}\left|R\varphi_k(x,\theta)-R\varphi_k(y,\theta')\right|^2\right]\\
				&\leq \frac{2^{2\alpha}C^2}{\sqrt{N}}C_{1,T}\left(1+\E[a_0^{2\alpha}]+T^{2\alpha}\right).\label{VVM-CLT-eq-def-2'}
				\end{aligned}
			\end{multline}
			Consequently from \eqref{VVM-CLT-eq-def-1'''} ,\eqref{VVM-CLT-eq-def-1''}, \eqref{VVM-CLT-eq-def-2'} and Assumption~\ref{VVM-CLT-ass-a_0}, it follows that,
			\begin{equation}\label{VVM-CLT-eq-def-3'}
				\sup_{N\geq 1}\sup_{0\leq t\leq T} \E\left[\sum_{k\geq 1}\Xi^N_{k,3}(t)^2\right]<\infty.
			\end{equation}
			From \eqref{VVM-CLT-eq-mu-0'}, \eqref{VVM-CLT-eq-def-0'} and \eqref{VVM-CLT-eq-def-3'} we deduce \eqref{CLT-VVM-eqR_phi}.
			
		\end{proof}
		
		\begin{lemma}\label{CLT-lem-V}
			Under Assumption~\ref{V-As-1}, Assumption~\ref{V-As-2} and Assumption~\ref{VVM-CLT-ass-a_0}, for all $T\geq 0,$ and $\alpha>\frac{1}{2}$,
			\begin{equation}
				\sup_N\sup_{0\leq t\leq T}\E \left[\|V^N_t\|_{-m_d,\alpha}^2\right]<\infty.
			\end{equation}
		\end{lemma}
		\begin{proof}
			Let $(\varphi_k)_{k\geq 1}$ be an orthonormal basis of $\mathcal{W}^{m_d,\alpha}_0$. By Parseval's identity
			\begin{align*}
				\|V^N_t\|_{-m_d,\alpha}^2&=\sum_{k\geq 1}\left(V^N_t(\varphi_k)\right)^2\\
				&\leq 2\sum_{k\geq 1}\left(\langle\hat{\mu}^N_t,\lambda\rangle^2\langle\mu^N_t, R\varphi_k\rangle^2+\langle\mu_t,\lambda\rangle^2\langle\hat\mu^N_t, R\varphi_k\rangle^2\right)\\
				&\leq2\langle\hat{\mu}^N_t,\lambda\rangle^2\sum_{k\geq 1}\left(\frac{1}{N}\sum_{j=1}^NR\varphi_k(a^N_j(t),\theta^N_j(t))^2\right)+2\langle\mu_t,\lambda\rangle^2\sum_{k\geq 1}\langle\hat\mu^N_t, R\varphi_k\rangle^2.
			\end{align*}
			As $\langle\mu_t,\lambda\rangle\leq\lambda_*,$ from Lemma~\ref{CLT-VVM-lem-vf},
			\begin{equation*}
				\sup_{N\geq 1}\sup_{0\leq t\leq T}\E\left[\langle\mu_t,\lambda\rangle^2\sum_{k\geq 1}\langle\hat\mu^N_t, R\varphi_k\rangle^2\right]<\infty.
			\end{equation*}
			To conclude, it remains to show that
			\begin{equation}\label{eq-hF}
				\sup_{N\geq 1}\sup_{0\leq t\leq T}\E\left[\langle\hat{\mu}^N_t,\lambda\rangle^2\sum_{k\geq 1}\left(\frac{1}{N}\sum_{j=1}^NR\varphi_k(a^N_j(t),\theta^N_j(t))^2\right)\right]<\infty.
			\end{equation} 
			As the family $(a_j(t),\theta_j(t))_j$ are exchangeable, from the expression of $R\varphi_k$ and the fact that $\gamma\leq1$ (Assumption~\ref{V-As-1}) and  $\int_{\Theta}K(\cdot,\wt\theta)\dP(d\wt\theta)=1$ (Assumption~\ref{V-As-2}), it follows that

				\begin{align*}
				&\E\left[\langle\hat{\mu}^N_t,\lambda\rangle^2\sum_{k\geq 1}\left(\frac{1}{N}\sum_{j=1}^NR\varphi_k(a^N_j(t),\theta^N_j(t))^2\right)\right]\\
				&\quad=\E\left[\langle\hat{\mu}^N_t,\lambda\rangle^2\sum_{k\geq 1}R\varphi_k(a^N_1(t),\theta^N_1(t))^2\right]\\
				&\quad\leq\E\left[\langle\hat{\mu}^N_t,\lambda\rangle^2\sup_{y\leq a_0^1+T,\theta,\theta'\in\Theta}\sum_{k\geq 1}\left|\varphi_k(0,\theta)-\varphi_k(y,\theta')\right|^2\right]\\
				&\quad=\E\left[\langle\hat{\mu}^N_t,\lambda\rangle^2\sup_{y\leq a_0^1+T,\theta,\theta'\in\Theta}\left\|D_{(0,\theta),(y,\theta')}\right\|_{-m_d,\alpha}^2\right]\\
				&\quad\leq2^{2\alpha}C^2\E\left[\langle\hat{\mu}^N_t,\lambda\rangle^2(1+a_0^{2\alpha}+T^{2\alpha})\right]\\
				&\quad=2^{2\alpha}C^2\E\left[(1+a_0^{2\alpha}+T^{2\alpha})\E\left[\langle\hat{\mu}^N_t,\lambda\rangle^2\big|a_0\right]\right]\\
				&\quad\leq2^{2\alpha}C^2C_{2,T}(1+\E[a_0^{2\alpha}]+T^{2\alpha}),
				\end{align*}
			where the last lines follow from Lemma~\ref{CLT-VVM-lem-inqf} and Proposition~\ref{VVM-CLT-prop-inq} given that $\langle\hat{\mu}^N_t,\lambda\rangle=\sqrt{N}\left(\fF^N(t)-\fF(t)\right)$.
			
			This concludes \eqref{eq-hF}.
		\end{proof}
		\begin{remark}\label{CLT-rq-th-0}
			Note that for $\varphi\in\mathcal{W}_0^{m_d+1,\alpha},\,\|L(\varphi)\|_{m_d,\alpha}=\|\partial_a\varphi\|_{m_d,\alpha}\leq\|\varphi\|_{m_d+1,\alpha}.$ This implies that for all $\phi\in\mathcal{W}_0^{-(m_d+1),\alpha},\,\|L^*(\phi)\|_{-(m_d+1),\alpha}\leq\|\phi\|_{-m_d,\alpha},$ where $L^*$ is the adjoint operator of $L$.

			Indeed, \[\|L^*(\phi)\|_{-(m_d+1),\alpha}\leq\|\phi\|_{-m_d,\alpha}\sup_{\|\varphi\|_{m_d+1,\alpha}=1}\|L(\varphi)\|_{m_d,\alpha}\leq\|\phi\|_{-m_d,\alpha}.\]
		\end{remark}
		We may consider the following decomposition in $\mathcal{W}_0^{m_d+1,\alpha}$,
		\begin{equation}\label{CLT-expr-d}
			\hat{\mu}^N_t=\hat{\mu}^N_0+\int_0^t L^*(\hat{\mu}^N_s)ds+\int_0^t V^N_s ds+ W^N_t.
		\end{equation}

		\begin{prop}\label{CLT-tight-VVM}
			Under Assumption~\ref{V-As-1}, Assumption~\ref{V-As-2} and Assumption~\ref{VVM-CLT-ass-a_0}, for every $\alpha>\frac{1}{2},$ the sequence of the laws of $(W^N)_N$ is tight in the space $D(\R_+,\mathcal{W}_0^{-m_d,\alpha})$ and $(\hat{\mu}^N)_N$ in the space $D(\R_+,\mathcal{W}_0^{-(m_d+1),\alpha})$. 
		\end{prop}
		\begin{proof}
			$(W^N)_N$ is tight if and only if $Tr(\ll W^N\gg)$ is tight, where $Tr(\ll W^N\gg)$ is the trace of $\ll W^N\gg$ given by \eqref{CLT-eq-W_q} (see \cite[page~$40$]{joffe1986weak}). 
			
			We have 
			\begin{equation*}
				Tr(\ll W^N\gg)_t=\sum_{k\geq 1}\int_{0}^{t}\langle \mu_s^N,\lambda\rangle\langle \mu_s^N,R^{(2)}\varphi_k\rangle ds.
			\end{equation*}
			However for $r\leq T,$ as $\gamma\leq 1,$ from Lemma~\ref{CLT-VVM-lem-inqf}, we have,
			\begin{align*}
				\sum_{k\geq 1}\langle \mu_r^N,R^{(2)}\varphi_k\rangle&=\frac{1}{N}\sum_{k\geq 1}\sum_{j=1}^N R^{(2)}\varphi_k(a^N_j(r),\theta_j^N(r))\\
				&=\frac{1}{N}\sum_{k\geq 1}\sum_{j=1}^N \int_{\Theta}\left(\varphi_k(0,\wt\theta)-\varphi_k(a^N_j(r),\theta_j^N(r))\right)^2\gamma(a^N_j(r),\theta_j^N(r))K(\theta^N_j(r),\wt\theta)\nu(d\wt\theta)\\
				&\leq \frac{1}{N}\sum_{j=1}^N \int_{\Theta}\|D_{(0,\wt\theta),(a^N_j(r),\theta_j^N(r))}\|_{-m_d,\alpha}^2K(\theta^N_j(r),\wt\theta)\nu(d\wt\theta)\\
				&\leq 2^{2\alpha}C^2\left(1+\frac{1}{N}\sum_{j=1}^N(a_0^j)^{2\alpha}+T^{2\alpha}\right).
			\end{align*}
			Consequently for all stopping time $(\tau^N) $ such that $ \tau^N\leq T, 0\leq \delta\leq\delta',$
			\begin{align*}
				\left|Tr(\ll W^N\gg)_{\tau^N+\delta}-Tr(\ll W^N\gg)_{\tau^N}\right|&=\left|\int_{\tau^N}^{\tau^N+\delta}\langle \mu_r^N,\lambda\rangle\sum_{k\geq 1}\langle \mu_r^N,R^{(2)}\varphi_k\rangle dr\right|\nonumber\\
				&\leq 2^{2\alpha}\lambda_*\delta' C^2\left(1+\frac{1}{N}\sum_{j=1}^N(a_0^j)^{2\alpha}+T^{2\alpha}\right)
			\end{align*}
			where the last line follows from the fact that $\langle \mu_r^N,\lambda\rangle\leq\lambda_*$.
			
			Therefore
			\begin{equation}\label{CLT-VVM-in-M-tight}
				\sup_N\sup_{\delta\leq\delta'}\E \left[\left|Tr(\ll W^N\gg)_{\tau^N+\delta}-Tr(\ll W^N\gg)_{\tau^N}\right|\right]\leq 2^{2\alpha}\lambda_*\delta' C^2(1+\E[a^{2\alpha}_0]+T^{2\alpha}).
			\end{equation}
			Gathering \eqref{CLT-VVM-in-M} with \eqref{CLT-VVM-in-M-tight} it follows that the sequence of the laws of $(W^N)_N$ is tight in $D(\R_+,\mathcal{W}_0^{-m_d,\alpha})$ according to Definition~\ref{Def-tight}.
			
			On the other hand, to prove the tightness of $(\hat{\mu}^N)_N$, as $(W^N)_N$ is tight, from \eqref{CLT-expr-d} it remains to prove that $\int_0^t L^*(\hat{\mu}^N_s)ds+\int_0^t V^N_s ds$ is tight.
			
			Indeed, we have 
			\begin{align*}
				\left\|\int_{\tau^N}^{\tau^N+\delta} \left[L^*(\hat{\mu}^N_s)+V^N_s\right] ds\right\|_{-(m_d+1),\alpha}^2&\leq  2\delta\int_{\tau^N}^{\tau^N+\delta}\left( \|L^*(\hat{\mu}^N_s)\|_{-(m_d+1),\alpha}^2+ \|V^N_s\|_{-(m_d+1),\alpha}^2\right)ds\\
				&\leq 2\delta'\int_{0}^{T+\delta'} \left(\|\hat{\mu}^N_s\|_{-m_d,\alpha}^2+ \|V^N_s\|_{-(m_d+1),\alpha}^2\right)ds,
			\end{align*}
			where the last line follows from Remark~\ref{CLT-rq-th-0}.
			
			Hence from Lemma~\ref{CLT-lem-eta}, and Lemma~\ref{CLT-lem-V}, it follows that,
			\begin{equation*}
				\sup_N\sup_{\delta\leq\delta'}\E\left[\left\|\int_{\tau^N}^{\tau^N+\delta} \left[L^*(\hat{\mu}^N_s)+V^N_s\right] ds\right\|_{-(m_d+1),\alpha}^2\right]< C\delta'(T+\delta').
			\end{equation*}
			This concludes the proof.
		\end{proof}
		
		\subsection{Characterization of the Limit of the subsequences }\label{sec-ch-limits}
			\begin{lemma}\label{CLT-boch-integr}
			Under Assumption~\ref{V-As-1}, Assumption~\ref{V-As-2} and Assumption~\ref{VVM-CLT-ass-a_0}, the integrals $\int_0^t L^*(\hat{\mu}^N_s)ds$ and $\int_0^t V^N_s ds$ are almost surely well defined as Bochner integrals in $\mathcal{W}_0^{-(m_d+1),\alpha}$ for any $\alpha>\frac{1}{2}$. In particular, the functions $t\mapsto \int_0^t L^*(\hat{\mu}^N_s)ds$ and $t\mapsto \int_0^t V^N_s ds$ are almost surely strongly continuous in $\mathcal{W}_0^{-(m_d+1),\alpha}$. 
		\end{lemma}
		\begin{proof} We adapt the proof of \cite[Lemma~$4.9$]{chevallier_fluctuations_2017}. 
			Since $\mathcal{W}_0^{-(m_d+1),\alpha}$ is separable, it suffices to verify that (see Yosida \cite[Theorem~$1$, page $133$]{yosida2012functional} ): 
			\begin{itemize}
				\item For every $\varphi\in\mathcal{W}_0^{m_d+1,\alpha},$ the functions $t\mapsto\langle L^*\hat{\mu}^N_t,\varphi\rangle=\langle\hat{\mu}^N_t,L(\varphi)\rangle$ and $t\mapsto V^N_t(\varphi)$ are measurable. 
				\item  The integrals $\int_0^t \|L^*(\hat{\mu}^N_s)\|_{-(m_d+1),\alpha}ds$ and $\int_0^t \|V^N_s\|_{-(m_d+1),\alpha}ds$ are finite almost surely.
			\end{itemize} 
			The proof for the first integral follows from Remark~\ref{CLT-rq-th-0}, and Lemma~\ref{CLT-lem-eta} and for the second integral from Lemma~\ref{CLT-lem-V}.  
		\end{proof}
		We will use the following Corollary to characterize the space of the limits of the subsequences.
		\begin{coro}\label{CLT-Prop-eta-in}
			Under Assumption~\ref{V-As-1}, Assumption~\ref{V-As-2} and Assumption~\ref{VVM-CLT-ass-a_0}, for all $T\geq 0,$ and $\alpha>\frac{1}{2},$
			\begin{equation}\label{CLT-prop-eta}
				\sup_N\E\left[\sup_{0\leq t\leq T}\|\hat{\mu}^N_{t}\|_{-(m_d+1),\alpha}^2\right]<\infty,
			\end{equation}
			and $t\mapsto \hat{\mu}^N_{t}$ belongs to $D(\R_+,\mathcal{W}_0^{-(m_d+1),\alpha})$.
		\end{coro}
		\begin{proof}
			From \eqref{CLT-expr-d} by convexity we have, 
			\begin{align*}
				\sup_{0\leq t\leq T}\|\hat{\mu}^N_{t}\|_{-(m_d+1),\alpha}^2&\leq 4\|\hat{\mu}^N_{0}\|_{-(m_d+1),\alpha}^2+4T\int_0^T \|L^*(\hat{\mu}^N_s)\|_{-(m_d+1),\alpha}^2ds\\&\hspace{2cm}+4T\int_0^T \|V^N_s)\|_{-(m_d+1),\alpha}^2ds+4\sup_{0\leq t\leq T}\|W^N_{t}\|_{-(m_d+1),\alpha}^2.
			\end{align*}
			As a result from Lemma~\ref{CLT-lem-eta}, Remark~\ref{CLT-rq-th-0}, Lemma~\ref{CLT-lem-V} and Lemma~\ref{CLT-mart-W}, we deduce \eqref{CLT-prop-eta}.
			
			The fact that $ \hat{\mu}^N_{t}$ is cadlag follows from the continuity of the integral in Lemma~\ref{CLT-boch-integr} and the fact that $W^N_t$ is cadlag in Lemma~\ref{CLT-mart-W}.
		\end{proof}
		
		The following Proposition states that all the limits of any converging subsequence, of $(W^N)_N$ belong to  $\mathcal{W}^{-m_d,\alpha}_0$ and are continuous, while those for $(\hat{\mu}^N)_N$ to $\mathcal{W}^{-(m_d+1),\alpha}_0$ are continuous.
		\begin{prop}\label{CLT-mart-w} For $\alpha>\frac{1}{2}$, every limit $\wt W$ of the sequence $(W^N)_N$ in $D(\R_+,\mathcal{W}_0^{-m_d,\alpha})$ and $\wt{\mu}$ of the sequence $(\hat{\mu}^N)_N$ in $D(\R_+,\mathcal{W}_0^{-(m_d+1),\alpha})$ satisfies
			\begin{equation}\E \left[\sup_{0\leq t\leq T}\|\wt W_t\|_{-m_d,\alpha}^2\right]<\infty\text{ and }\E \left[\sup_{0\leq t\leq T}\|\wt{\mu}_t\|_{-(m_d+1),\alpha}^2\right]<\infty.\label{CLT-c-tight}\end{equation}
			Moreover, the limit laws are continuous. 
		\end{prop}
		\begin{proof}
			From \cite[Theorem~$13.4$]{billingsley1999convergence} it suffices to prove that, for all $T\geq 0$, the maximal jump size of $W^N$ and $\hat{\mu}^N$ on $[0,T]$ converge to $0$ almost surely. So the result follows from the following inequality
			\begin{equation*}
				\sup_{\varphi\in\mathcal{C}^\infty_c,\,\|\varphi\|_{m_d+1,\alpha}=1}|\langle\hat{\mu}^N_t,\varphi\rangle-\langle\hat{\mu}^N_{t^-},\varphi\rangle|=\sup_{\varphi\in\mathcal{C}^\infty_c,\,\|\varphi\|_{m_d,\alpha}=1}|W^N_t(\varphi)-W^N_{t^-}(\varphi)|\leq\frac{C}{\sqrt{N}},
			\end{equation*}
			where we use the fact that, as $\alpha>\frac{1}{2},$ from \eqref{Injection-w}, $\mathcal{C}^{m_d}_b\subset \mathcal{W}^{m_d,\alpha}_0.$
			
			Consequently
			\[\sup_{0\leq t\leq T}\|\hat{\mu}^N_t-\hat{\mu}^N_{t^-}\|_{-(m_d+1),\alpha}=\sup_{0\leq t\leq T}\|W_t^N-W_{t^-}^N\|_{-m_d,\alpha}\leq\frac{C}{\sqrt{N}}.\]
			Inequality \eqref{CLT-c-tight} follows from Corollary~\ref{CLT-Prop-eta-in} and Lemma~\ref{CLT-mart-W} combining with the fact that the mapping $g\mapsto \sup_{0\leq t\leq T}\|g_t\|_{k,\alpha}$ from $D(\R_+,\mathcal{W}_0^{-k,\alpha})$ to $\R_+$ is continuous at every point in $\mathcal{C}(\R_+,\mathcal{W}_0^{-k,\alpha})$.
		\end{proof}
		\begin{prop}\label{V-As-w} Under Assumption~\ref{V-As-1}, Assumption~\ref{V-As-2} and Assumption~\ref{VVM-CLT-ass-a_0},
			for any $\alpha>\frac{1}{2}$, The sequence $(W^N)_N$ converges in law in $\mathcal{C}(\R_+,\mathcal{W}_0^{-(m_d+1),\alpha})$ toward a centered continuous Gaussian process $W$ given by Definition~\ref{CLTM-def1}.
		\end{prop}
		\begin{proof}
			As $(W^N)_N$ is $\bC$-tight in $D(\R_+,\mathcal{W}_0^{-(m_d+1),\alpha}),$ we can extract a subsequence denoted again $(W^N)_N$ that converges to $\wt W$ in $\mathcal{C}(\R_+,\mathcal{W}_0^{-(m_d+1),\alpha})$. It is easy to check that $\wt W$ is martingale (see for example \cite[Proposition~$4.5.1$]{tran2006modeles}). From Proposition~\ref{CLT-mart-w}, $\wt W$ is a square integrable martingale such that for $\varphi\in\mathcal{W}_0^{m_d+1,\alpha},$
			\[\lim_{N\to\infty}\ll W^N \gg_t(\varphi)=\ll\wt W \gg_t(\varphi),\]
			where we recall that $\ll W^N \gg$ is given by the following expression in \eqref{CLT-eq-W_q}.
			\begin{equation*}
				\ll W^N\gg_t(\varphi)=\int_{0}^{t}\langle \mu_s^N,\lambda\rangle\langle \mu_s^N,R^{(2)}\varphi\rangle ds.
			\end{equation*}
			We set for $g\in D(\R_+;\mathcal{P}(\R_+\times\Theta)),$
			\[H_\varphi(g)(t)=\int_{0}^{t}\langle g_s,\lambda\rangle\langle g_s,R^{(2)}\varphi\rangle ds.\]
			Noting that, for $\alpha>\frac{1}{2},$ from \eqref{Injection-w}, $\mathcal{C}^{m_d+1}_b\subset \mathcal{W}^{m_d+1,\alpha}_0.$ Consequently, from expression of $R^{(2)}\varphi$ in \eqref{CLT-eq-W_q} and Assumption~\ref{V-As-1}, it follows that, there exists $C>0,$ for all $(a,\theta)\in\R_+\times\Theta,$
			\begin{align*}
				|R^{(2)}\varphi(a,\theta)|&\leq\int_{\Theta}\left(\varphi(0,\wt\theta)-\varphi(a,\theta)\right)^2K(\theta,\wt\theta)\nu(d\wt\theta)\\
				&\leq2\|\varphi\|_{\mathcal{C}^{m_d+1}_b}\leq2C^2\|\varphi\|_{m_d+1,\alpha}^2
			\end{align*}  
			where the last line follows from Assumption~\ref{V-As-2}.
			
			As a result, it is easy to check that,
			\[\sup_{0\leq t\leq T}|H_\varphi(g^1)(t)-H_\varphi(g^2)(t)|\leq 4TC^2\lambda_*\|\varphi\|_{m_d+1,\alpha}^2\sup_{0\leq t\leq T}\|g^1_t-g^2_t\|_{TV}.\]
			So from continuous mapping theorem and Theorem~\ref{VVM-th} it follows that $H_\varphi(\mu^N)\to H_\varphi(\mu)$ as $N\to\infty$. So $\wt W$ is a continuous square integrable martingale with a deterministic Doob-Meyer process, so it is characterised as the Gaussian process with covariance given by \eqref{CLT-cov-w} and the uniqueness follows. 
		\end{proof}
		\begin{prop}\label{CLT-prop-existence}
			Under Assumption~\ref{V-As-1}, Assumption~\ref{V-As-2} and Assumption~\ref{VVM-CLT-ass-a_0}, Given $W$ a centered Gaussian process given by Definition~\ref{CLTM-def1}, the following equation, for $\alpha>\frac{1}{2}$ and any $\varphi\in \mathcal{W}^{m_d+2,\alpha}_0$
			\begin{equation}\label{CLT-expr-d-limit-1}
				\langle \hat{\mu}_t,\varphi\rangle=\langle \hat{\mu}_0,\varphi\rangle+\int_{0}^{t}\langle \hat{\mu}_s,L(\varphi)\rangle ds+\int_{0}^{t}\left(\langle \hat{\mu}_s,\lambda\rangle\langle \mu_s,R\varphi\rangle+\langle \mu_s,\lambda\rangle \langle \hat{\mu}_s,R\varphi\rangle\right)ds + W_t(\varphi),	\end{equation}
			has at most one solution $\hat{\mu}\in \mathcal{C}(\R_+,\mathcal{W}_0^{-(m_d+1),\alpha})$. 				 
		\end{prop}
		\begin{proof} For $\varphi\in\mathcal{W}_0^{m_d+2,\alpha},$ 
			We define
			\begin{equation*}
				\forall t\in\R_+,\forall s\in[0,t],\forall (a,\theta)\in\R_+\times\Theta,\quad f(s,a,\theta)=\varphi(a-(s-t),\theta).
			\end{equation*}
			It is easy to check that, $f$ is the unique solution of the following parametric transport equation
			\begin{equation*}
				\left\{\begin{aligned}
					&\partial_s f_s(a,\theta)+\partial_a f_s(a,\theta)=0\quad\forall s\in[0,t]\\
					&f(t,a,\theta)=\varphi(a,\theta).
				\end{aligned}\right.
			\end{equation*} 
			Therefore from \eqref{CLT-exp-VVM},
			\begin{equation*}
				\langle \hat{\mu}_t,\varphi\rangle=\langle \hat{\mu}_0,\varphi_t\rangle+\int_{0}^{t}\left(\langle \hat{\mu}_s,\lambda\rangle\langle \mu_s,R\varphi_{t-s}\rangle+\langle \mu_s,\lambda\rangle \langle \hat{\mu}_s,R\varphi_{t-s}\rangle\right)ds + W_t(\varphi),	
				\end{equation*}
				where $\varphi_s(a,\theta)=\varphi( a+s,\theta)$.
				
				Let $\hat{\mu}^1$ and $\hat{\mu}^1$ two solutions of equation~\eqref{CLT-exp-VVM} in $\mathcal{C}(\R_+,\mathcal{W}_0^{-(m_d+1),\alpha})$. Consequently,
				\begin{equation}\label{CLT-expr-d-limit-1-p1}
					\langle \hat{\mu}_t^1-\hat{\mu}_t^2,\varphi\rangle=\int_{0}^{t}\left(\langle \hat{\mu}_s^1-\hat{\mu}_s^2,\lambda\rangle\langle \mu_s,R\varphi_{t-s}\rangle+\langle \mu_s,\lambda\rangle \langle \hat{\mu}_s^1-\hat{\mu}_s^2,R\varphi_{t-s}\rangle\right)ds,	
				\end{equation}
			We recall that  
			\begin{equation*}
				R\varphi(a,\theta)=\int_{\Theta}\left(\varphi(0,\widetilde\theta)-\varphi(a,\theta)\right)\gamma( a,\theta)K(\theta,\widetilde\theta)\dP(\rd\widetilde\theta).
			\end{equation*}
			Note that as $\alpha>\frac{1}{2}$, from \eqref{Injection-w}, $\mathcal{C}^{m_d+2}_b\subset \mathcal{W}^{m_d+2,\alpha}_0.$
			For $\varphi\in\mathcal{C}^{m_d+2}_b,$ as $\gamma\leq1$ (Assumption~\ref{V-As-1}), we have for all $(a,\theta)\in\R_+\times\Theta,$ 
			\begin{align*}
				|R\varphi_{t-s}(a,\theta)|&\leq\int_{\Theta}\left|\varphi_{t-s}(0,\widetilde\theta)-\varphi_{t-s}(a,\theta)\right|\gamma( a,\theta)K(\theta,\widetilde\theta)\dP(\rd\widetilde\theta)\\
				&\leq2\|\varphi\|_{\mathcal{C}^{m_d+2}_b},
			\end{align*}
			where the last line follows from Assumption~\ref{V-As-2}.
			
			As a result, as $\alpha>\frac{1}{2}$, from \eqref{Injection-w}, there exists $C>0$ such that,
			 \begin{equation}\label{eq-inw}
			 	\|R\varphi_{t-s}\|_{m_d+2,\alpha}\leq2C\|\varphi\|_{m_d+2,\alpha}\text{ and }\|\lambda R\varphi_{t-s}\|_{m_d+2,\alpha}\leq2C\lambda_*\|\varphi\|_{m_d+2,\alpha}.
			 \end{equation}
			Note that as $\mathcal{C}^{m_d+2}_b$ is dense in $\mathcal{W}^{m_d+2,\alpha}_0,$ inequality \eqref{eq-inw} holds on $\mathcal{W}^{m_d+2,\alpha}_0.$
			
			Therefore from \eqref{CLT-expr-d-limit-1-p1} and the fact that $\langle u_s,\lambda\rangle\leq\lambda_*$, it follows that,
			\begin{equation*}
				\|\hat{\mu}^1_t-\hat{\mu}^2_t\|_{-(m_d+2),\alpha}\leq 4C\lambda_*\int_{0}^{t}\|\hat{\mu}^1_s-\hat{\mu}^2_s\|_{-(m_d+2),\alpha}ds.
			\end{equation*}
			Then the result follows from Gronwall's Lemma.
		\end{proof}
		\begin{remark}
			Note that $\partial_a\varphi$ appearing in \eqref{CLT-expr-C} and \eqref{CLT-expr-d-limit-1} reduces the regularity of the test functions by $1$. So if we consider $\varphi\in\mathcal{W}_0^{m_d+1,\alpha}$ we mut consider $\hat{\mu}^N_t\in \mathcal{W}_0^{-m_d,\alpha}$ when dealing with this term $\int_{0}^{t}\langle \hat{\mu}_s^N,\partial_a\varphi\rangle ds$. Moreover $\hat{\mu}^N_t$ is not tight in $ \mathcal{W}_0^{-m_d,\alpha}$, so we will consider $\varphi\in\mathcal{W}_0^{m_d+2,\alpha}$.
		\end{remark}
		\subsection{Proof of Theorem~\ref{CLT-formulation-VVM}}\label{CLT-Main-r}
			From Proposition~\ref{CLT-tight-VVM}, $(\hat{\mu}^N)_N$ is $\bC$-tight in $D(\R_+,\mathcal{W}_0^{-(m_d+1),\alpha}),$ hence we can extract a subsequence denoted again $(\hat{\mu}^N)_N$ that converges to $\wt \mu$ in $\mathcal{C}(\R_+,\mathcal{W}_0^{-(m_d+1),\alpha})$. Moreover as $\hat{\mu}^N$ satisfies \eqref{CLT-expr-C} from Proposition~\ref{V-As-w} and continuous mapping theorem's it follows that $\wt\mu$ satisfies \eqref{CLT-expr-d-limit-1}. Uniqueness follows from Proposition~\ref{CLT-prop-existence}. This concludes the proof.

       \bibliographystyle{plain}
       \bibliography{biblio}

		\appendix
		\section{}\label{App}
		\subsection{Definition of tightness}
		\begin{definition}\label{Def-tight}Let $H$ be a separable Hilbert space.
			A sequence of stochastic process $(X^N)_N\in D(\R_+,H)$ is tight if the following conditions hold:
			\begin{itemize}
				\item For all $t\geq0$ and $\epsilon>0$ there exists a compact $\mathcal{K}$ such that,   
				\begin{equation}\label{Ald-1}
					\sup_N\mathbb{P}(X^N_t\notin\mathcal{K})\leq\epsilon.
				\end{equation}
				\item for all $\epsilon_1>0,\,\epsilon_2>0,$ and $T>0,$ there exist $\delta'>0$ and an integer $N_0$ such that for all stopping time $\tau_N\leq T,$
				\begin{equation}
					\sup_{N\geq N_0}\sup_{\delta\leq\delta'}\mathbb{P}\left(\|X^N_{\tau_N+\delta}-X^N_{\tau_N}\|_H\geq\epsilon_1\right)\leq\epsilon_2.
				\end{equation}
			\end{itemize}  
		\end{definition}
		Note that the following condition implies condition~\eqref{Ald-1}: There exists a Hilbert space $H_0$ such that $H_0\hookrightarrow_K H$ and for all $t\geq0,$
		\begin{equation*}
			\sup_N\E\left[\|X^N_t\|_{H_0}^2\right]<\infty,
		\end{equation*} 
		where the notation $\hookrightarrow_K$ means that the injection is compact.
\subsection{Explicit expression of the random variable $(\mathfrak{M}_{0,1},\mathfrak{M}_{0,2},\mathfrak{M}_1,\mathfrak{M}_2)$}\label{A-sec-N}
To obtain their expressions, it suffices to compute $\langle \hat{u}_t,\varphi\rangle$. More precisely, by replacing $\hat{u}_t$ by the expression \eqref{VVL-eq-CLT-2'} in $\langle \hat{u}_t,\varphi\rangle$, we easily obtain the following expression,
		\begin{enumerate}
			\item $\mathfrak{M}_{0,1}(\varphi)(t):=\int_{\R_+\times\Theta}\varphi(a+t,\theta)\hat{u}_0(a,\theta)\exp\left(-\int_{0}^{t}\fF(s)\gamma(a+s,\theta)\rd s\right)da\dP(d\theta)$
			\item 
			\begin{multline*}
				\mathfrak{M}_{0,2}(\varphi)(t):=\int_{\R_+\times\Theta}\varphi(a+t,\theta)\int_0^t \sqrt{\fF(s)\gamma(a+s,\theta)u_0(a,\theta)}\zeta_s(a+s,\theta)\\\times\exp\left(-\int_{s}^{t}\fF(r)\gamma(r+a,\theta)\rd r-\frac{1}{2}\int_{0}^{s}\fF(r)\gamma(r+a,\theta)\rd r\right)ds da\dP(d\theta)
			\end{multline*}			
			\item \begin{multline*}
				\mathfrak{M}_{1}(\varphi)(t):=\int_{\Theta}\int_{0}^t\varphi(t-s,\theta)\int_{\R_+\times\Theta}\sqrt{\fF(s)\gamma(a,\wt\theta)u_s(a,\wt\theta)}\zeta_s(a,\wt\theta)K(\widetilde\theta,\theta)\rd a\dP(\rd\widetilde{\theta})\\\exp\left(-\int_{s}^{t}\fF(r)\gamma(r-s,\theta)\rd r\right)ds\dP(d\theta)
			\end{multline*}
			\item \begin{multline*}
				\mathfrak{M}_{2}(\varphi)(t):=\int_{\Theta}\int_{0}^t\varphi(t-a,\theta)\int_a^{t} \sqrt{\fF(s)\fF(a)\fS(a,\theta)\gamma(s-a,\theta)}\zeta_{s}(s-a,\theta)\\\exp\left(-\int_{s}^{t}\fF(r)\gamma(r-a,\theta)\rd r-\frac{1}{2}\int_{a}^{s}\fF(r)\gamma(r-a,\theta)\rd r\right)ds da\dP(d\theta).
			\end{multline*}
		\end{enumerate}

	
\end{document}